\documentclass[final,4p,times]{elsarticle}
\usepackage{enumerate}
\usepackage{amsmath}
\usepackage{amssymb}
\usepackage{amsthm}
\usepackage{amsfonts}
\usepackage{microtype}
\usepackage{mathrsfs}
\usepackage{graphicx}
\usepackage{color}
\usepackage{latexsym}
\usepackage{rotating}
\usepackage{xspace}
\usepackage[all]{xy}
\usepackage{longtable}
\usepackage{leftidx}
\usepackage{mathtools}
\usepackage{tikz}
\usepackage{geometry}

\usepackage{calc}
\usepackage{tikz}
\usetikzlibrary{decorations.markings}
\tikzstyle{vertex}=[circle, draw, inner sep=0pt, minimum size=6pt]
\newcommand{\vertex}{\node[vertex]}

\newtheorem{thm}{Theorem}[section]
\newtheorem{cor}[thm]{Corollary}
\newtheorem{lem}[thm]{Lemma}

\newtheorem{dfn}[thm]{Definition}
\newtheorem{pro}[thm]{Proposition}
\newtheorem{rem}[thm]{Remark}

\newtheorem{exa}[thm]{Example}

=21cm\headheight=0cm\topskip=0cm\topmargin=0cm

\def\e{\mathbf {e}}

\begin{document}

\begin{frontmatter}

\title{On partial cubes, well-graded families and their duals with some applications in graphs}

\author{Alireza Mofidi}

\address{
{\footnotesize Department of Mathematics and Computer Science,}
\\{\footnotesize Amirkabir University of Technology (Tehran Polytechnic), Iran}

and

{\footnotesize School of Mathematics, Institute for Research in Fundamental Sciences {\rm (IPM)}}
\\{\footnotesize{\it P.O. Box {\rm 19395-5746} Tehran, Iran.\\e-mail: mofidi@aut.ac.ir}}
}

\begin{abstract}
Well-graded families, extremal systems and maximum systems (the last two in the sense of VC-theory and Sauer-Shelah lemma on VC-dimension) are three important classes of set systems. 
This paper aims to study the notion of duality in the context of these classes of set systems and then use the obtained results for studying graphs.
More specifically, we are concerned with  
the characterization of the finite set systems 
which themselves and their dual systems are both well-graded, extremal or maximum.
On the way to this goal, 
and maybe also of independent interest,
we study the structure of the well-graded families with the property that the size of the system is not much bigger than the size of
its essential domain, 
that is, the set of elements of the domain which are shattered by the system as single element subsets. 
As another target of the paper, we use the above results to characterize graphs 
whose set systems of open or closed neighbourhoods, cliques or independent sets are well-graded, extremal or maximum.
We clarify the relation of such graphs to the celebrated half-graphs.
Through the paper, we frequently relate our investigations to the VC-dimension of the systems.
Also we use one-inclusion graphs associated to set systems as an important technical tool. 
\end{abstract}

\begin{keyword}

Well-graded families \sep partial cubes \sep dual of set systems \sep VC-dimension \sep 
extremal systems \sep maximum systems.

{\sc MSC codes: 05C75 \sep 05C65 \sep 05D05 \sep 05C20 \sep 03C45}

\end{keyword}

\end{frontmatter}

\section{Introduction}

The general theme of this paper is to study certain interactions between notions of duality in set systems and the properties such as well-gradedness, being maximum or extremal.
We also frequently relate these notions to the VC-dimension of the systems. Moreover, we use the result we obtain for set systems for studying graphs.
The notions of well-gradedness and duality 
are of special interest in this paper and play crucial roles.

Partial cubes are graphs that admit isometric embeddings into hypercube graphs.
These graphs were introduced by Graham and Pollak \cite{GrahamPollak} for modelling interconnection networks 
and were extensively studied afterwards in several papers 
such as \cite{Chepoi,Djokovic,EppsteinThelatticedimension,KlavzarShpectorovTribesofcubicpartialcubes,OvchinnikovPartialcubesStructurescharacterizations,WinklerIsometricembeddingsinproductsofcompletegraphs}.
Nowadays, the investigation of hypercubes and partial cubes 
has found many applications in different parts of combinatorics and established links to other fields of mathematics as well as computer science.
From another point of view, there is a well known class of set systems, called well-graded families (or isometrically-embedded systems), which are defined as the systems whose one-inclusion graphs (see Definition \ref{defoneincgraph}) are partial cube. 
Note that the study of set systems (or hypergraphs) is one of the core areas of research focus in combinatorics.
Well-graded families 
are of special interest in this paper.
The notion of well-gradedness 
was introduced in \cite{DoignonFalmagnewgfamiliesofrelations} and was extensively studied later.
For a comprehensive account of known results on partial cubes and well-graded families, the interested reader can refer to for example the books 
\cite{bookmediatheory}, \cite{BookImrichKlavzarProductGraphs} and \cite{bookpartialcubes}.
One can see that partial cubes and well-graded families are two faces of the same coin. So, partial cubes can be naturally studied via well-graded families and vice versa and this could be considered as one of the reasons of significance of well-graded families.
Another reason to consider well-graded families as important set systems might rely on the very fact that they contain the well-known class of extremal systems (Definition \ref{defextremalsystems}), which in turn include the class of maximum systems (Definition \ref{defmaximumsystems}). Numerous nice examples of set systems stemming from 
geometrical configurations such as hyperplane-arrangements lie into these classes.  
Maximum set systems are the ones with the largest possible cardinality
for their given VC-dimension. 
Such classes are of particular
interest in discrete geometry and learning theory.
Similarly, 
extremal systems are also of particular interest in learning theory and subjects such as construction of compression schemes (see for example \cite{MornWarmuthLabeledCompSch}).
Another reason of importance of partial cubes and well-graded families is their connections to a very general and broad programme called "metric embeddings", which roughly speaking, deals with the problems of how good different metric spaces can be embedded into each other 
(see Chapter 15 of \cite{MatousekLeconDiscGeom} for some combinatorial approaches to the metric embedding problems).
More precisely, 
partial cubes by definition are instances of metric spaces (with distance metric on graphs) which are nicely (isometrically) embedded in certain important metric spaces, namely hypercubes.
This viewpoint might enable one to see partial cubes and well-graded families from the lens of the metric embedding problems.

During the course of our investigation,
we mainly pursue three goals.
First and foremost, we are interested in the characterization of well-graded families, extremal systems and maximum systems whose duals are again well-graded, extremal and maximum respectively. 
Note that the dual system of a set system is a consequential system associated to it which encodes various structural properties of that system. 
Second, we provide 
the characterization of particular classes of well-graded families possessing certain property which we will call additionality 1 or 2. This result, 
which may also be of independent interest,
serves as an introductory step for the characterization mentioned in the first goal.
Third, we are concerned with employing the above mentioned results
to study certain set systems associated to graphs from the point of view of the properties of well-gradedness, being extremal or being maximum. For instance, we characterize all graphs whose set systems of neighbourhoods are well-graded, extremal or maximum and show that they are closely related to the well-known half-graphs. 
Note that the half-graphs are important combinatorial objects which appear in graph theory and some other fields of mathematics. For example, in model theory, and more precisely Shelah's classification theory (see \cite{bookshelahclassification}), a very important class of theories introduced by Shelah, called stable theories, are defined, roughly speaking, by  
the property that
their formulas are not able to
encode arbitrary large half-graphs (or order property in the terminology of model theory) in the models of the theory. 
In light of this connection to model theory, one can see that our above-mentioned result on graphs can naturally connect us to logic and find a model theoretic meaning. However, in the present paper, we will only have a combinatorial point of view and leave pursuing the model theoretical aspects to a further work.
In another direction, it is worth mentioning that in the context of celebrated Szemeredi regularity lemma, many interesting connections between the notion of "irregular-pairs" 
and 
half-graphs have been discovered in combinatorics.
The interested reader can refer to for example \cite{MalliarisShelahReglemforstablegraphs} where the half-graphs and the model theoretic order property are involved 
for analysing the Szemeredi regularity lemma.

One-inclusion graphs of set systems appear in this paper as an important technical tool
for analysing properties such as well-gradedness in set systems and their dual counterparts. 
Such graphs
have been frequently used as central notions and substantial 
tools to scrutinize set systems in several works in the literature. 

Organization of the paper is as follows.
In Section \ref{sectionpreliminaries}, we first review in Subsection \ref{reviewpartcubesandrelatedsystems} some essential notions from the literature and previous works that we need in the rest of the paper such as well-graded families, extremal systems, maximum systems, etc. Then, in Subsection \ref{basicdefnotationfacts}, we introduce some new notions and prove some facts which we will need in the forthcoming sections.

In Section \ref{strofpartcubewithsmalladditionalities}, we 
define the notion of additionality of a set system as the difference between the size of the system and the number of elements of its domain which are shattered as one element subsets. 
Then, utilizing one-inclusion graphs, we characterize well-graded families with additionalities 1 and 2. 
In fact, we prove a handful of properties of well-graded families and their one-inclusion graphs (such as Proposition \ref{FessEGFinequalities}, Lemma \ref{sizeofbelow} 
and characterization theorems \ref{characterizeextvc=1withequalities,right=impliesleft=} and \ref{structureofpathdensewithF=ess+2}) where apart from their own interests, help us in the different characterizations we have in mind to settle later on in this work 
such as the main theorem of Section \ref{sectioncharselfanddualpartcube}.
Note that 
we also set forth a few characterizations of well-graded families of VC-dimension 1 in Theorem \ref{characterizeextvc=1withequalities,right=impliesleft=}; meanwhile some more characterizations of extremal systems of VC-dimension $1$
which were established before in \cite{MeszarosRonyaismallvc} are presented.
Moreover, as another result involving VC-dimension, we show that the VC-dimension of well-graded families with additionality $2$ is equal to $2$.

In Section \ref{sectioncharselfanddualpartcube}, we state and prove Theorem \ref{characterizationselfanddualpathdense}, which is one of the central goals of this paper. This theorem characterizes well-graded families 
whose duals 
are well-graded as well.
The proof of this theorem heavily relies on several results proved in this and former sections in particular characterization of set systems with additionalities 1 and 2 which was completely studied in Section \ref{strofpartcubewithsmalladditionalities}.
We further derive a few more readily obtainable characterizations of extremal (maximum) systems 
whose duals are also extremal (maximum).

Finally, in Section \ref{partcubeofsetsystemsassociatedtographs}, equipped with our 
acquired knowledge through the previous sections on well-graded families (in particular Theorem \ref{characterizationselfanddualpathdense}), we proceed to explore the properties such as well-gradedness, being extremal or being maximum for some of the set systems associated to graphs such as the systems of open and closed neighbourhoods, set system of cliques and set system of the independent sets. Half-graphs come to the picture in these investigations.

\section{Preliminaries}\label{sectionpreliminaries}

\subsection{A review on well-graded families and related set systems}\label{reviewpartcubesandrelatedsystems}
In this subsection we review some definitions and facts from the literature and fix some notations we will need in the the paper.

By a \textit{set system} (or hypergraph) $(X,\mathcal{F})$ we mean a set $X$, which is called the domain of the set system, and a nonempty family $\mathcal{F}$ of distinct subsets of $X$.
Through the paper, we assume that the domains of the set systems under consideration are finite. 
We call two set systems $(X,\mathcal{F})$ and $(Y,\mathcal{G})$ isomorphic if there exists some bijection $f:X \rightarrow Y$ such that for every $A \subseteq X$, we have $A \in \mathcal{F}$ if and only if $f(A) \in \mathcal{G}$. Note that an equivalent way to define the isomorphism is that there exists a bijection $f$ 
such that $\mathcal{G}=\{f(A):A \in \mathcal{F}\}$.
In a graph $G$, we denote the set of vertices and edges by $V(G)$ and $E(G)$ respectively. If $G$ is connected, one can define a distance between any two vertices $v$ and $w$ of $G$, and denote it by $d_G(v,w)$, as the number of edges in a shortest path between $v$ and $w$. One can see that $d_G$ is a metric on the set $V(G)$.
In this paper we call a directed graph connected when its underlying undirected graph is connected. Also we consider the distance between two vertices in a directed graph as their distance in the underlying undirected graph.
By the degree of a vertex in a directed graph we mean the degree of that vertex in its underlying undirected graph. Also by incoming (outgoing) degree of a vertex  we mean the number of edges incoming to (outgoing from) that vertex. 
In an edge-labelled graph we denote the label of an edge $e$ by $lab(e)$.

\vspace{2mm}

\noindent \textbf{Well-graded families and one-inclusion graphs}

\vspace{1.5mm}

\begin{dfn}\label{defoneincgraph}
By the \textit{one-inclusion graph} of a set system $(X,\mathcal{F})$ we mean a simple edge-labelled directed graph $G_{\mathcal{F}}$ whose vertices are in a one-to-one correspondence to members of $\mathcal{F}$, namely each element $A \in \mathcal{F}$ corresponds to a vertex of $G_{\mathcal{F}}$ which we denote it by $v_A$. Moreover, for any $A, B \in \mathcal{F}$ there is a directed edge $e$ from $v_A$ to $v_B$ if and only if there exists some $b \in X$ such that $B =A \cup \{b\}$.
Also the label on this directed edge 
is $b$ and is referred by the notation $lab(e)$.
For every vertex $v \in V(G_{\mathcal{F}})$ we denote the member of $\mathcal{F}$ corresponding to $v$ by $st(v)$.
\end{dfn}

One-inclusion graphs are essential tools for studying set systems and have been used as important technical tools in many papers such as 
\cite{GrecoEmbeddingstracefinitesets},
\cite{HausslerLittlestoneWarmuthPredictingfunctionsonrandomlydrawnpoints},
\cite{KuzminWarmuth}, \cite{MeszarosRonyaismallvc}, \cite{MornWarmuthLabeledCompSch},
and \cite{RubinsteinRubinsteinBartlettShiftingOneinclusionmistake}.
The following statement is a well-known fact and can be verified easily.
Note that we call a directed graph bipartite when its underlying undirected graph is bipartite.

\begin{rem}\label{1incgraphisbipartite}
The one-inclusion graph of every set system is bipartite.
\end{rem}

For any two sets $A,B$ in a set system, we define $d_h(A,B)=|A \triangle B|$ and call it the Hamming distance of $A, B$.
In this paper we consider a well-known class of set systems called well-graded families defined as follows.

\begin{dfn}\label{defisomembedinfinitesystems}
We call a set system $(X,\mathcal{F})$ 
a \textit{well-graded family} (or well-graded system) 
if for every $A, B \in \mathcal{F}$ we have
$d_h(A,B)=d_{G_{\mathcal{F}}}(v_A,v_B).$
\end{dfn}

Well-graded families are also known as \textit{isometrically embedded systems} in the literature.
One sees that the one-inclusion graph of every well-graded family is connected.

\begin{dfn}
For any natural number $n$, by the $n$-dimensional hypercube graph we mean a graph 
whose vertices corresponding to binary $\{0,1\}$-vectors of length $n$ where two vertices are connected if and only if the corresponding vectors 
differ in exactly one coordinate.
\end{dfn}

\begin{dfn}
A graph is a called a \textit{partial cube} if and only if it can be isometrically embedded in a hypercube.
\end{dfn}
The above definition can be equivalently stated in this way that a graph $G$ is a partial cube if and only if its vertices can be labelled by subsets of a fix set $X$ in such a way that the distance between any two vertices in the graph is equal to the size of the symmetric difference of
the labels of those vertices. 
\begin{rem}
An undirected graph is a partial cube if and only if it is isomorphic to the underlying undirected graph of the one-inclusion graph of a well-graded family.
\end{rem}

By the above remark, partial cubes can be naturally studied via well-graded families and vice versa.
Partial cubes have 
rich structural properties. 
These graphs were introduced by Graham and Pollak in \cite{GrahamPollak}.
Note that the characterizations of Djokovic, \cite{Djokovic}, Winkler \cite{WinklerIsometricembeddingsinproductsofcompletegraphs} and Chepoi \cite{Chepoi} and many other fundamental results in \cite{Chepoi,Djokovic,EppsteinThelatticedimension,OvchinnikovPartialcubesStructurescharacterizations,WinklerIsometricembeddingsinproductsofcompletegraphs}
have given us a good understanding of the structure of partial cubes.
The notion of well-graded families
was introduced in \cite{DoignonFalmagnewgfamiliesofrelations}.
Several other investigations concerning partial cubes and well-graded families have been carried out
in numerous papers such as 
\cite{AlbenqueKnauerConvexityinpartialcubes,BresarImrichKlavzarFastrecognitionpartialcubesDiscAppMath,EppsteinThelatticedimension,FukudaHanda, KlavzarMulderPartialcubesandcrossinggraphs,KlavzarShpectorovTribesofcubicpartialcubes,KlavzarShpectorovConvexexcesspartialcubes,OvchinnikovPartialcubesStructurescharacterizations, 
WilkeitIsometricembeddingsinHamminggraphs,
WinklerIsometricembeddingsinproductsofcompletegraphs}.
Also interactions of these notions with media theory were explored in 
several works such as \cite{EppsteinFalmagneAlgorithmsformediaDiscAppMath}, \cite{FalmagneOvchinnikovMediaTheoryDiscAppMath}, \cite{OvchinnikovMediaTheoryRepandExmDiscAppMath}. 
Partial cubes and well-graded families have many applications in different parts of combinatorics and established links to other fields of mathematics as well as theoretical computer science. 
Also partial cubes have found many applications outside of mathematics. For example, papers such as \cite{KlavzaGutmanMoharLabelingBenzenoid}, \cite{KlavzarNadjafiAraniComputingdistanceDjokovicWinklerrelationDiscApplMath} or \cite{ZhangXuNonecoronoidembeddedintohypercubeDiscApplMath} study certain relations between 
partial cubeness and properties of some particular classes of graphs with application in the natural sciences.
For a comprehensive account of known results and related papers on partial cubes and well-graded families, the interested reader can refer the books 
\cite{BookDezaLaurentGeometryofCutsandMetrics, bookmediatheory,BookImrichKlavzarProductGraphs} and \cite{bookpartialcubes}.

\vspace{1mm}

The following is a simple observation about well-graded families.

\begin{rem}\label{rempathdenseforinfiniteandfinitesystems}
A set system $(X,\mathcal{F})$ 
is a well-graded family if and only if for every $A, B \in \mathcal{F}$ with 
$|A \bigtriangleup B|\geqslant 2$
there exists some $C \in \mathcal{F}$ such that $C \bigtriangleup A \subsetneqq A \bigtriangleup B$ and 
$C \bigtriangleup B \subsetneqq A \bigtriangleup B$ (or equivalently 
$A \cap B \subseteq C \subseteq A \cup B$ and $C \not=A$ and $C \not = B$). 
\end{rem}

One of the reasons for importance of class of well-graded families, besides its tight connections to partial cubes, 
arises from the fact that it contains many important subclasses such as the class of maximum systems and the class of extremal systems which we define below.

\vspace{2mm}

\noindent \textbf{Maximum and extremal systems}

\vspace{2mm}

For a set system $(X,\mathcal{F})$ and $Y \subseteq X$ define
$\mathcal{F} \cap Y:=\{A \cap Y: A \in \mathcal{F}\}.$
We call the new set system $\mathcal{F} \cap Y$ on the domain $Y$ the \textit{trace} of the set system $\mathcal{F}$ on the set $Y$.
In a set system $(X,\mathcal{F})$ a subset $Y \subseteq X$ is called \textit{shattered} by $\mathcal{F}$ if $\mathcal{F} \cap Y=\mathcal{P}(Y)$.
The VC-dimension of $\mathcal{F}$, denoted by $VCdim(\mathcal{F})$, is the largest integer $n$ 
such that there exists some 
subset of $X$ of size $n$ which is shattered by $\mathcal{F}$.
One can find the following important theorem in \cite{SauerOnthedensityoffamiliesofsets} or \cite{ShelahAcombinatorialproblemstabilityandorderformodelsandtheories}.

\begin{thm}[Sauer-Shelah lemma]\label{SauerShelahLemma}
Assume that $(X,\mathcal{F})$ is a set system with $VCdim(\mathcal{F})=d$.
Then for every 
$Y \subseteq X$ we have
$$|\mathcal{F} \cap Y| \leqslant \sum_{i=0}^{d}\binom{|Y|}{i}.$$
\end{thm}

\begin{dfn}\label{defmaximumsystems}(Maximum systems)
We call a set system $(X,\mathcal{F})$ with $VCdim(\mathcal{F})=d$ a \textit{$d$-maximum} system if for any $Y \subseteq X$, the inequality of the Sauer-Shelah lemma turns to equality. Also we call a set system maximum if it is $d$-maximum for some non-negative integer $d$.
\end{dfn}
Maximum systems are of particular interest in discrete geometry and learning theory. One can refer to for example \cite{KuzminWarmuth} and \cite{RubinsteinRubinsteinBartlettBoundingEmbeddings} to get a glimpse of some of the focal results on maximum systems and applications of this class in learning theory.

By a \textit{shattering-cube} on $Y \subseteq X$ (or a $Y$-shattering-cube) we mean a set system $\mathcal{C}$ on $X$ of the form 
$\mathcal{C}=\{Z \cup T: Z \subseteq Y\}$ for some fix $T \subseteq X \setminus Y$ which is called the \textit{tag} of the shattering-cube $\mathcal{C}$.
In a set system $(X,\mathcal{F})$, the subset $Y \subseteq X$ is said to be \textit{strongly shattered} by $\mathcal{F}$ if $\mathcal{F}$ contains a $Y$-shattering-cube.
By $sht(\mathcal{F})$ and $ssht(\mathcal{F})$ we mean the set of all subsets of $X$ shattered and strongly shattered by $\mathcal{F}$ respectively. The following theorem is an important result about set systems and was proven independently in several papers. One can refer to \cite{MoranShatteringExtremalSystemsMastersthesis} to see a list of references for this theorem. 

\begin{thm}[Sandwich theorem] 
Let $(X,\mathcal{F})$ be a set system. 
Then  $|ssht(\mathcal{F})| \leqslant |\mathcal{F}| \leqslant |sht(\mathcal{F})|.$
\end{thm}

It is worth mentioning that an easy consequence of the right side inequality of the sandwich theorem implies the Sauer-Shelah lemma (see Section 1 of \cite{AnsteeRonyaiSaliShatteringnews}).
We define another important class of set systems 
as the systems for which the inequalities of the sandwich theorem 
turn into equality. 

\begin{dfn}\label{defextremalsystems}
A system $(X,\mathcal{F})$ is called \textit{extremal} if $ssht(\mathcal{F})=sht(\mathcal{F})$.
\end{dfn}

If in a set system, the right inequality of the sandwich theorem turns to equality, then the system is called \textit{s-extremal}. It is shown in \cite{BollobasRadcliffe} and \cite{BandeltCombinatoricsoflopsidedsets} that if either of the inequalities of the sandwich theorem turns to equality, then the other one turns to equality too. In fact, s-extremal systems are the same as the extremal systems.
Extremal systems were discovered independently several times in different papers such as \cite{LawrenceLopsidedsets} 
and \cite{BollobasRadcliffe}. 
They were extensively studied and several characterizations of them were given in 
\cite{AnsteeRonyaiSaliShatteringnews,
BandeltCombinatoricsoflopsidedsets,
BollobasRadcliffe,
GrecoEmbeddingstracefinitesets,
LawrenceLopsidedsets,
MoranShatteringExtremalSystemsMastersthesis,
RonyaiMeszarosSomecombinatorialapplicationsofGrobnerbases}.
Also many other interesting related results in connection with combinatorial structure of these systems have been obtained in papers such as \cite{BollobasLeaderRadcliffeReverseKleitmanInequalities}, 
\cite{FurediQuinnTraceofFiniteSets} and \cite{KozmaMoranShatteringGraphorientation}.
The following remark is a known fact about relation between maximum and extremal systems.
\begin{rem}\label{maximumsareextremal}
Every maximum set system 
is extremal. 
\end{rem}

Note that the other direction of the Remark \ref{maximumsareextremal} does not hold and there are extremal systems which are not maximum. 
The classes of maximum and extremal systems contain many important set systems arising from combinatorial and computational geometry such as the hyperplane-arrangements. 
Also they have very good connections to several problems in learning theory. 
The interested reader can refer to for example \cite{MornWarmuthLabeledCompSch} or \cite{KuzminWarmuth} to see some examples and further details about appearance of these classes of systems in different areas.
One of the sources of importance of class of well-graded families might be because of the following result due to Greco 
which shows that maximum systems and extremal systems are both well-graded families.

\begin{rem}\label{Grecohamdisequalsgraphdis}
(See \cite{GrecoEmbeddingstracefinitesets})
Every extremal set system is a well-graded family.
\end{rem}

Therefore, the class of well-graded families is a rich class containing the classes of maximum and extremal systems and much more.
We denote the class of maximum systems, extremal systems and well-graded families by $MAX, EXT$ and $WG$ respectively.
As a summary of remarks \ref{maximumsareextremal} and \ref{Grecohamdisequalsgraphdis}, we have
$MAX \subseteq EXT \subseteq WG.$ 

\subsection{Some basic definitions, notations and facts}\label{basicdefnotationfacts}
In this part we introduce some notions we need in the next sections and prove some statements about them.

\begin{dfn}\label{defFbarreverse}
For a set system $(X,\mathcal{F})$ by $(X,\overline{\mathcal{F}})$ we mean a set system on $X$ with $\overline{\mathcal{F}}=\{A^c:A \in \mathcal{F}\}$.
\end{dfn}
\begin{dfn}
By the \textit{reverse} of an edge-labelled directed graph we mean a graph with the same vertex set, edges and labels with the only difference that the directions of the edges are reversed.
\end{dfn}

The statements of the following remarks are easy to be verified.

\begin{rem}\label{reverseofoneincgraph}
The reverse of the one-inclusion graph of a set system $(X,\mathcal{F})$ is the one-inclusion graph of the system $(X,\overline{\mathcal{F}})$.
\end{rem}

\begin{rem}\label{complementofpartcubeispartcube}
A set system $(X,\mathcal{F})$ is a well-graded family, an extremal system or a maximum system if and only if 
$(X,\overline{\mathcal{F}})$ is a well-graded family, an extremal system or a maximum system respectively.
\end{rem}

\begin{dfn}
In a set system $(X,\mathcal{F})$ we call a $x \in X$ \textit{essential} if there are $A,B \in \mathcal{F}$ such that $x \in A$ and $x \not \in B$ or equivalently if $\{x\}$ as a one element subset of $X$ is shattered by $\mathcal{F}$. We denote the subset of essential elements of $X$ by $ess_{\mathcal{F}}(X)$ and call it the \textit{essential domain} of the system $(X,\mathcal{F})$. 
Also for every $Y \subseteq X$, we define $ess_{\mathcal{F}}(Y):=ess_{\mathcal{F}}(X) \cap Y$.
We call the quantity 
$|\mathcal{F}|-|ess_{\mathcal{F}}(X)|$ the \textit{additionality} of the system $(X,\mathcal{F})$. 
\end{dfn}

The following remarks are easy to be verified.

\begin{rem} Extending the domain of a system by non-essential elements does not affect the properties of being extremal or being well-graded for the system, but can affect the property of being maximum. More precisely, extending the domain of a maximum set system with positive VC-dimension by non-essential elements gives rise to a non-maximum system.
\end{rem}

\begin{rem}\label{|ess(X)|numberofdistinctlabels}
Let $(X,\mathcal{F})$ be a set system 
whose one-inclusion graph is connected (for example well-graded families).
Then every $x \in ess_{\mathcal{F}}(X)$ appears as the label of at least one edge of the graph $G_{\mathcal{F}}$. Moreover, the label of every edge of $G_{\mathcal{F}}$ is an element of $ess_{\mathcal{F}}(X)$. Therefore, $|ess_{\mathcal{F}}(X)|$ is equal to the number of distinct labels in $G_{\mathcal{F}}$.

\end{rem}

\begin{dfn}\label{defflipoperation}
In a set system $(X,\mathcal{F})$ and for any $a \in X$, by the \textit{bit-flip} operation in $\mathcal{F}$ on element $a$ we mean the set system $\mathcal{F}'$ defined on $X$ as 
$\mathcal{F}':=\{A \setminus \{a\}: A \in \mathcal{F}, a \in A\} \cup \{A \cup \{a\}: A \in \mathcal{F}, a \not \in A\}.$
Also for any $A \in \mathcal{F}$ and $B \subseteq X$, 
by $flip(\mathcal{F}, A \rightarrow B)$ we mean a set system obtained from $\mathcal{F}$ by applying a bit-flip operation for every $a \in A \triangle B$.
We call such operations the flip operations.
\end{dfn}

\begin{rem}\label{infliptoemtryeverymemberissubsetofess}
Let $(X,\mathcal{F})$ be a set system, $A \in \mathcal{F}$ and $\mathcal{F}':=flip(\mathcal{F}, A \rightarrow \emptyset)$.
Then we have $ess_{\mathcal{F}'}(X)=ess_{\mathcal{F}}(X)$.
Also $(X,\mathcal{F})$ and $(X,\mathcal{F}')$ have the same additionality.
Moreover, every member of $\mathcal{F}'$, viewed as a subset of $X$, is a subset of $ess_{\mathcal{F}}(X)$. 
\end{rem}

\begin{rem}\label{VCdimofflipissameastheoriginalsystem}

Let $(X,\mathcal{F})$ be a system, $A,B \in \mathcal{F}$ and $\mathcal{F}':=flip(\mathcal{F}, A \rightarrow B)$.
Then $VCdim(\mathcal{F})=VCdim(\mathcal{F}')$.
\end{rem}
\begin{proof}
One can see that the shattered subsets of $X$ by $\mathcal{F}$ and $\mathcal{F}'$ are the same.
It follows the result. \hfill $\square$
\end{proof}

\begin{rem}\label{inWGeveryelemofshatteredsetappearslabel}
Let $(X,\mathcal{F})$ be a well-graded family and assume that $D \subseteq X$ is shattered by $\mathcal{F}$. Then every $a \in D$ appears as the label of at least $2^{|D|-1}$ edges of $G_{\mathcal{F}}$.
\end{rem}
\begin{proof}
Fix an arbitrary $a \in D$. For every $U \subseteq D$ containing $a$, since $D$ is shattered by $\mathcal{F}$, we can find $K_U$ and $L_U$ in $\mathcal{F}$ such that $K_U \cap D=U$ and $L_U \cap D=U \setminus \{a\}$.
Since the system is well-graded, there is a path $P_U$ in $G_{\mathcal{F}}$ between $v_{K_U}$ and $v_{L_U}$ of length $|K_U \triangle L_U|$ such that the labels of its edges are elements of $K_U \triangle L_U$ each one used once. So, there is an edge, say $e_U$, 
in $P_U$ with label $a$. Hence, every $U \subseteq D$ containing $a$ gives rise to an edge $e_U$ with label $a$ in $G_{\mathcal{F}}$. Thus, it is enough to show that different such $U$'s give rise to different $e_U$'s.
It is not hard to see that for every such $U$'s and every vertex $v_A$ in $P_U$, we have 
$A \cap (D \setminus \{a\})=U \setminus \{a\}$.
So for every different $U, U' \subseteq D$ containing $a$, two paths $P_U$ and $P_{U'}$ do not have any common vertex which follows that $e_U$ and $e_U'$ are distinct.
\hfill $\square$
\end{proof}

\vspace{3mm}

\noindent \textbf{Dual systems and well-gradedness}

\vspace{2mm}

The notion of dual systems defined as follows is an important tool in studying set systems and has been frequently used in the literature.
The interested reader can see for example Chapter 17 of the book \cite{bookBergegraphsandhypergraphs} for some concepts around duality.

\begin{dfn}\label{defdualsystem}
By \textit{dual system} of a set system $(X,\mathcal{F})$, denoted by $(X,\mathcal{F})^*$, we mean a set system $(Y_{\mathcal{F}},\mathcal{F}^*)$ where $Y_{\mathcal{F}}=\{y_A:A \in \mathcal{F}\}$ where for each $A \in \mathcal{F}$, $y_A$ is a symbol corresponding to $A$ and $y_A$'s are different for different $A$'s, and $\mathcal{F}^*=\{\mathcal{A}_x:x \in X\}$ where for every $x \in X$, we define $\mathcal{A}_x:=\{y_A: A \in \mathcal{F}, x \in A\}$. 
If instead of $\mathcal{F}^*$ we consider its subset $\mathcal{F}^*_{ess}:=\{\mathcal{A}_x:x \in ess_{\mathcal{F}}(X)\}$ as the system on $Y_{\mathcal{F}}$, then we call this system the 
\textit{essential dual} (denoted by ess-dual) of $(X,\mathcal{F})$. We call a set system \textit{self-dual} (\textit{self-ess-dual}) if it is isomorphic to its dual (ess-dual) system.
\end{dfn}

In this paper and in the definition of the dual systems, we consider $\mathcal{F}^*$ as a set and not a multiset. It means that if for $x,y \in X$ we have $\mathcal{A}_x=\mathcal{A}_y$, then $\mathcal{A}_x$ and $\mathcal{A}_y$ are counted both as one member of $\mathcal{F}^*$ and not two.
Duality of set systems can be also explained nicely with the language of matrices.
By incidence matrix of a set system, we mean a matrix that has one column corresponding to each member $A$ of the system and one row corresponding to each member $x$ of the domain of the system and the entry in the column corresponding to $A$ and row corresponding to $x$ is $1$ if $x \in A$ and $0$ otherwise.
In order to find the dual system of a set system, it is enough to consider the incidence matrix of the system, then transpose it, and finally identify the identical columns. Then, the obtained matrix would be the incidence matrix of the dual set system.

\begin{rem}\label{partcube,complement,dual}
Assume that $(X,\mathcal{F})$ is a set system. 
Then two set systems $(Y_{\mathcal{F}},\overline{\mathcal{F}^*})$ and $(Y_{\overline{\mathcal{F}}},{\overline{\mathcal{F}}}^*)$ are isomorphic via the map $y_A \rightarrow y_{A^c}$ for every $A \in \mathcal{F}$.
\end{rem}

\begin{rem}\label{identifyingY_Fwithsetofverticesof1-inc}
For a set system $(X,\mathcal{F})$, whenever we need we identify the elements of $Y_{\mathcal{F}}$ with the set of vertices of the graph $G_{\mathcal{F}}$ via the correspondence $y_A \leftrightarrow v_A$ for every $A \in \mathcal{F}$. This enables us to see $\mathcal{A}_x$'s, the members of the dual system, as subsets of the vertices of $G_{\mathcal{F}}$.
\end{rem}

For a graph $G$ and two disjoint subsets $A, B \subseteq V(G)$, we denote the set of edges between $A$ and $B$ by $E(A,B)$.
By a cut-set in a graph $G$ we mean a set of edges of the form $E(A,A^c)$ for some $A \subseteq V(G)$.
In any set system $(X,\mathcal{F})$ and for every $a \in ess_{\mathcal{F}}(X)$, we denote the set of edges with label $a$ in $G_{\mathcal{F}}$ by $C_a$.
In the following statement we relate the dual of the set systems with the cut-sets of the one-inclusion graphs of them.

\begin{rem}\label{inpartcubedgeswithlabelaiscutsets}
If $(X,\mathcal{F})$ is a set system, then for every $a \in ess_{\mathcal{F}}(X)$, $C_a$ is the same as the cut-set $E(\mathcal{A}_a,\mathcal{A}_a^c)$ of the graph $G_{\mathcal{F}}$ 
(where $\mathcal{A}_a$ is viewed as subsets of the vertices of $G_{\mathcal{F}}$ as explained in Remark \ref{identifyingY_Fwithsetofverticesof1-inc}) and directions of the edges in $C_a$ is from $\mathcal{A}_a^c$ to $\mathcal{A}_a$. Moreover, if the system is well-graded, then subgraphs induced on both $\mathcal{A}_a$ and $\mathcal{A}_a^c$ are connected.
\end{rem}
\begin{proof}
Fix some $a \in ess_{\mathcal{F}}(X)$.
Every edge $e \in E(\mathcal{A}_a,\mathcal{A}_a^c)$ is connecting a vertex $v_A$ to a vertex $v_B$ for some $A,B \in \mathcal{F}$ with $a \not \in A$ and $a \in B$.
So By definition of one-inclusion graph $B=A \cup \{a\}$, the edge has a direction from $v_A$ to $v_B$ and the label on it is $a$. So $e \in C_a$. Hence $E(\mathcal{A}_a,\mathcal{A}_a^c) \subseteq C_a$. On the other hand, since every edge $e \in C_a$ has label $a$, the member of $\mathcal{F}$ corresponding to one of ends of $e$ contains $a$ and the other one not. 
So one of two ends of $e$ belongs to $\mathcal{A}_a$ and the other end not. 
Thus, $e$ is an edge between $\mathcal{A}_a$ and $\mathcal{A}_a^c$. 
Therefore, $C_a \subseteq  E(\mathcal{A}_a,\mathcal{A}_a^c)$.
So far we have shown that $C_a=E(\mathcal{A}_a,\mathcal{A}_a^c)$.
Now it is also clear that direction of every edge in $C_a$ is from $\mathcal{A}_a^c$ to $\mathcal{A}_a$.
For every two vertices $v_A$ and $v_B$ in $\mathcal{A}_a$, both $A$ and $B$ contain $a$. So if $(X,\mathcal{F})$ is moreover a well-graded family, then there is a path with $|A \triangle B|$ edges in $G_{\mathcal{F}}$ between $v_A$ and $v_B$ and since $a \not \in A \triangle B$, no edge of this path has $a$ as its label. Therefore, this path does not pass through the set of edges $C_a$ (which is shown to be the same as $E(\mathcal{A}_a,\mathcal{A}_a^c)$) and is entirely inside the subgraph induced on $\mathcal{A}_a$ by $G_{\mathcal{F}}$. So this subgraph is connected. A similar argument shows that the subgraph induced on $\mathcal{A}_a^c$ is also connected.
\hfill $\square$
\end{proof}

\vspace{1.5mm}

By using Remark \ref{inpartcubedgeswithlabelaiscutsets}, one can see that the members of the dual system $(Y_{\mathcal{F}},\mathcal{F}^*)$ (except possibly two members $\emptyset$ and $Y_{\mathcal{F}}$ which correspond to elements of $X \setminus ess_{\mathcal{F}}(X)$) correspond to the cut-sets of the graph $G_{\mathcal{F}}$ of the form $C_a$ for some $a \in ess_{\mathcal{F}}(X)$. Roughly speaking, certain cut-sets of the graph $G_{\mathcal{F}}$ determine the dual system.

One of the goals of this paper is to study the well-graded families 
whose dual systems are also well-graded. We first give a name to such systems in the following definition.

\begin{dfn}\label{defselftanddualextsystem}
We call a set system $(X,\mathcal{F})$ 
"\textit{self-and-(ess-) dual well-graded}"
if the system itself and also its (ess-) dual system are both well-graded families. The notions of "\textit{self-and-(ess-) dual maximum}" and "\textit{self-and-(ess-) dual extremal}" are also defined in a similar way only by replacing the property of well-gradedness by being maximum and being extremal respectively.
\end{dfn}

We define a system to be \textit{(ess-) dual well-graded family} if its (ess-) dual system is a well-graded family.

\begin{rem}\label{selfanddualpartcubeiffcomplementselfanddualpartcube}
A system $(X,\mathcal{F})$ is a dual well-graded family if and only if 
$(X,\overline{\mathcal{F}})$ is a dual well-graded family.
Similarly, $(X,\mathcal{F})$ is a self-and-dual well-graded family if and only if $(X,\overline{\mathcal{F}})$ is self-and-dual well-graded.
\end{rem}
\begin{proof}
If $(X,\mathcal{F})$ is a dual well-graded family, $(Y_{\mathcal{F}},\mathcal{F}^*)$ is a well-graded family. So by Remark 
\ref{complementofpartcubeispartcube}, $(Y_{\mathcal{F}},\overline{\mathcal{F}^*})$ is a well-graded family. Now by Remark \ref{partcube,complement,dual},
$(Y_{\overline{\mathcal{F}}},\overline{\mathcal{F}}^*)$ is a well-graded family which follows that $(X,\overline{\mathcal{F}})$ is a dual well-graded family. 
Conversely, if $(X,\overline{\mathcal{F}})$ is a dual well-graded family, then the same argument and the fact that $(X,\overline{\overline{\mathcal{F}}})=(X,\mathcal{F})$ completes the proof.
In the case that $(X,\mathcal{F})$ is self-and-dual well-graded,
again by 
Remark \ref{complementofpartcubeispartcube}
$(X,\overline{\mathcal{F}})$ is well-graded. So combining this with the previous argument, $(X,\overline{\mathcal{F}})$ is self-and-dual well-graded. The converse is also clear by using the similar arguments as above.
\hfill $\square$
\end{proof}

\vspace{1.5mm}

By using remarks \ref{maximumsareextremal} and \ref{Grecohamdisequalsgraphdis}, one can see that every self-and-(ess-) dual maximum or extremal system is self-and-(ess-) dual well-graded.

\vspace{4mm}

\noindent \textbf{Purification of set systems}

\begin{dfn}\label{eqrelhavingsameFtype}
Let $(X,\mathcal{F})$ be a set system. We say that two elements $x,y \in X$ have the same $\mathcal{F}$-types, denoted by $x \sim_{\mathcal{F}} y$, if for every $A \in \mathcal{F}$ we have $x \in A$ if and only if $y \in A$.
\end{dfn}
One can see that the relation $\sim_{\mathcal{F}}$ is an equivalence relation on $X$. The system $\mathcal{F}$ induces a system $\tilde{\mathcal{F}}$ on $\frac{X}{\sim_{\mathcal{F}}}$ in a natural canonical way and there is a canonical surjective set system homomorphism from $(X,\mathcal{F})$ to $(\frac{X}{\sim_{\mathcal{F}}},\tilde{\mathcal{F}})$ where a surjective homomorphism between two set systems is a surjective map from the domain of the first system to the domain of the second one such that the image of every member of the first system under the map is a member of the second system.

\begin{dfn}\label{systemquotientofERsameFtype}
We call the set system $(\frac{X}{\sim_{\mathcal{F}}},\tilde{\mathcal{F}})$ obtained from $(X,\mathcal{F})$ in the above way the \textit{purification} of the system $(X,\mathcal{F})$. We call a set system \textit{purified} if it is equal to its purification.
\end{dfn}

\begin{dfn}\label{defalmostselfdual}
We call a set system $(X,\mathcal{F})$ \textit{almost self-dual} 
if $(X,\mathcal{F})^*$, the dual system of $(X,\mathcal{F})$, and $(\frac{X}{\sim_{\mathcal{F}}},\tilde{\mathcal{F}})$, the purification of the system $(X,\mathcal{F})$ defined in Definition \ref{systemquotientofERsameFtype}, are isomorphic to each other. 
\end{dfn}

\begin{rem}\label{pathdensex,ysametypeoutsideess}
Assume that $(X,\mathcal{F})$ is a well-graded family and $x$ and $y$ are distinct elements of $X$ with the same $\mathcal{F}$-types. Then $x,y \in X \setminus ess_{\mathcal{F}}(X)$.
In particular, if 
$X=ess_{\mathcal{F}}(X)$, then the system is purified.
\end{rem}
\begin{proof}
Clearly an element of $ess_{\mathcal{F}}(X)$ cannot have the same $\mathcal{F}$-type as an element in $X \setminus ess_{\mathcal{F}}(X)$.
So it is enough to show that every $x,y \in ess_{\mathcal{F}}(X)$ do not have the same $\mathcal{F}$-types. Assume for contradiction that $x,y \in ess_{\mathcal{F}}(X)$ have the same $\mathcal{F}$-types. 
Let $A,B \in \mathcal{F}$ be such that $x \in A$ and $x \not \in B$. 
So $y \in A$ and $y \not \in B$ too.
Hence, $x,y \in A \bigtriangleup B$. 
Since $(X,\mathcal{F})$ is a well-graded family, one sees that there is some $C \in \mathcal{F}$ 
containing exactly one of $x$ and $y$. But this contradicts with the assumption that $x$ and $y$ have the same $\mathcal{F}$-types.
\hfill $\square$ 
\end{proof}

\begin{rem}\label{seconddualisomorphtoERsameFtype}
Let $(X,\mathcal{F})$ be a set system. Then $(X,\mathcal{F})^{**}$, the dual of dual of the system (which is also called the second dual of the system), is isomorphic to the purification of 
$(X,\mathcal{F})$ namely $(\frac{X}{\sim},\tilde{\mathcal{F}})$. 
So a set system is isomorphic to its second dual if and only if it is purified.
\end{rem}

\begin{cor}
Every well-graded family $(X,\mathcal{F})$ with $X=ess_{\mathcal{F}}(X)$ is isomorphic to its second dual.
\end{cor}
\begin{proof}
By combination of remarks \ref{pathdensex,ysametypeoutsideess} and \ref{seconddualisomorphtoERsameFtype} the result is clear. \hfill $\square$
\end{proof}

\begin{cor}
Every self-dual system is almost self-dual.
\end{cor}
\begin{proof}
Let $(X,\mathcal{F})$ be a self-dual set system. Thus, $(X,\mathcal{F})^{**}$, the dual of dual of the system, is isomorphic to the system.
But by Remark \ref{seconddualisomorphtoERsameFtype}, $(X,\mathcal{F})^{**}$ is isomorphic to the purification of the system. So $(X,\mathcal{F})$ is purified. Hence, the dual system is isomorphic to the purification of the system. So $(X,\mathcal{F})$ is almost self-dual. \hfill $\square$
\end{proof}

\section{Structure of well-graded families with small additionalities}\label{strofpartcubewithsmalladditionalities}

The main goals of this section are theorems \ref{characterizeextvc=1withequalities,right=impliesleft=} and \ref{structureofpathdensewithF=ess+2} where we characterize well-graded families with small additionalities (additionalities 1 and 2) in terms of the properties of their one-inclusion graphs and also VC-dimension.
On the way to these goals, we study one-inclusion graphs of well-graded families and prove several properties of them. These properties help us to obtain the mentioned characterizations.
Results of this section will be also used later in Section \ref{sectioncharselfanddualpartcube}. 
We remind that when we say that a directed graph is connected, we mean that its underlying undirected graph is connected. Also we consider the distance between two vertices in a directed graph as their distance in the underlying undirected graph.

\begin{pro}\label{FessEGFinequalities}
Assume that $(X,\mathcal{F})$ is a well-graded family.
Then $|ess_{\mathcal{F}}(X)|+1 \leqslant |\mathcal{F}|\leqslant |E(G_{\mathcal{F}})|+1.$
\end{pro}
\begin{proof}
We first prove the right inequality. Since the system is well-graded, its one-inclusion graph $G_{\mathcal{F}}$ 
is connected. So $|E(G_{\mathcal{F}})| \geqslant |V(G_{\mathcal{F}})|-1=|\mathcal{F}|-1.$
Now we prove the left inequality. 
First we note that since the system is a well-graded family, by Remark \ref{|ess(X)|numberofdistinctlabels} every $x \in ess_{\mathcal{F}}(X)$ should appear as the label of at least one edge of $G_{\mathcal{F}}$.
Choose a subset $K$ of edges of $G_{\mathcal{F}}$ consisting of $|ess_{\mathcal{F}}(X)|$ edges with different labels. So every $x \in ess_{\mathcal{F}}(X)$ appears as the label of some edge in $K$. 
Now the set of edges in $K$ form a subgraph of $G_{\mathcal{F}}$. 
Ignore the directions of the edges of this subgraph and denote its underlying undirected graph by $G'$. 
Note that $G'$ might be disconnected. Also $|ess_{\mathcal{F}}(X)|=|K|=|E(G')|$. 
Since label of each edge of $G_{\mathcal{F}}$ is the single element in the symmetric difference of the members of $\mathcal{F}$ corresponding to its ends, it is not hard to see that in every cycle of $G_{\mathcal{F}}$, the number of appearance of every label is even. Therefore, since all labels of edges of $G'$ are different, $G'$ does not contain any cycle. So $G'$ is a tree or a forest.
Thus, 
$|ess_{\mathcal{F}}(X)|=|E(G')| < |V(G')| \leqslant |V(G_{\mathcal{F}})|= |\mathcal{F}|$.
It completes the proof. \hfill $\square$
\end{proof}

\begin{cor}
Every well-graded family has positive additionality.
\end{cor}
\begin{proof}
By Proposition \ref{FessEGFinequalities} we have
$|\mathcal{F}|-|ess_{\mathcal{F}}(X)| \geqslant 1$.
\hfill $\square$
\end{proof}
\begin{dfn}\label{dfnlabeless}
Let $(X,\mathcal{F})$ be a set system.
For every subset $U$ of vertices of $G_{\mathcal{F}}$ we define 
$$st(U):=\bigcup_{v \in U}st(v),$$
$$below(U):=\{v \in V(G_{\mathcal{F}}):\exists u \in U, \ st(v) \subseteq st(u)\}$$
where we recall from Definition \ref{defoneincgraph} that $st(v)$ is the member of $\mathcal{F}$ corresponding to $v$.
Also define $ess_{\mathcal{F}}(U):=ess_{\mathcal{F}}(\bigcup_{v \in U} st(v))$.
\end{dfn}

\begin{lem}\label{sizeofbelow}
Let $(X,\mathcal{F})$ be a well-graded family with $\emptyset \in \mathcal{F}$ and 
additionality $r$ for some $r \in \mathbb{N}$.
Let $U$ be a nonempty subset of vertices of $G_{\mathcal{F}}$. Then we have 
$|st(U)|+1 \leqslant |below(U)| \leqslant |st(U)|+r.$
\end{lem}
\begin{proof}
The inequalities can be easily verified for the case $U=\{v_{\emptyset}\}$.
We consider the other cases.
Let $H$ be the induced subgraph of $G_{\mathcal{F}}$ on the vertices of $below(U)$. Obviously $v_{\emptyset} \in below(U)$ and $U \subseteq below(U)$. Also since $\mathcal{F}$ is a well-graded family, for every $u \in below(U)$ different from $v_{\emptyset}$, there is a path between $v_{\emptyset}$ and $u$ such that for every vertex $z$ on the path we have $st(z) \subseteq st(u)$. But it implies that $z \in below(U)$. So the path is entirely inside the subgraph $H$.
It follows that $H$ is a connected subgraph. So there is a walk, denoted by W, in $H$ which starts from $v_{\emptyset}$ and passes through all vertices in $H$. 
Note that since $\emptyset \in \mathcal{F}$, $st(U)=ess_{\mathcal{F}}(U)$. 
Obviously, every $x \in st(U)$ appears in $st(v)$ for some vertex $v$ in $H$.
For every $x \in st(U)$, let $w_x$ be the vertex in $H$ such that $x$ appears in the set $st(w_x)$ for the first time since we started walking in $H$ on walk $W$ from its starting vertex $v_{\emptyset}$.

\vspace{1.5mm}

\textit{Claim.} For any two distinct $x,y \in st(U)$ we have $w_x \not =w_y$.

\vspace{1.5mm}

\textit{Proof of Claim.}
Assume for contradiction that $w_x=w_y$ for two distinct $x,y \in st(U)$. Let $c$ be the previous vertex of $w_x$ in the walk $W$ (since walk starts from $v_{\emptyset}$ and $w_x$ cannot be the same as $v_{\emptyset}$, there exists a previous vertex for $w_x$ in the walk). 
We have $x,y \in st(w_x)=st(w_y)$. But $x,y \not \in st(c)$ since $w_x$ is the first vertex in the order of the walk $W$ 
whose corresponding element of $\mathcal{F}$ contains $x$ and $y$. So $|st(c) \triangle st(w_x)|\geqslant 2$. 
But since $c$ and $w_x$ are connected, we have $|st(c) \triangle st(w_x)|=1$.
This is a contradiction.
\hfill \textit{Claim} $\square$

\vspace{1mm}

Since by the above claim $V(H)$ contains distinct $w_x$'s for distinct $x$'s in $st(U)$ and also it contains $v_{\emptyset}$ (which is different from all such $w_x$'s), then 
$|V(H)| \geqslant |st(U)|+1$.
Therefore, since $V(H)=below(U)$, we have $|st(U)|+1 \leqslant |below(U)|$.

Now we prove the second inequality in a similar way. Since $G_{\mathcal{F}}$ is connected, there is a walk $W'$ in $G_{\mathcal{F}}$ which starts from $v_{\emptyset}$ and passes through all vertices of it. Define $w_x$'s in a similar way as above for every 
$x \in ess_{\mathcal{F}}(X)$. Again similar to the above claim, for any two $x,y \in ess_{\mathcal{F}}(X)$ we have $w_x \not =w_y$. 
In particular, for any two $x,y \in ess_{\mathcal{F}}(X) \setminus st(U)$ we have $w_x \not =w_y$.
Moreover, for any $x \in ess_{\mathcal{F}}(X) \setminus st(U)$, we have 
$w_x \not \in below(U)$ since otherwise $st(w_x) \subseteq st(v) \subseteq st(U)$ for some $v \in U$, while
$x \in st(w_x)$ (by definition of $w_x$) but $x \not \in st(U)$ which is a contradiction.
Hence, for any $x \in ess_{\mathcal{F}}(X) \setminus st(U)$, $w_x$
is not a vertex of $H$.
So there are at least $|ess_{\mathcal{F}}(X)|-|st(U)|$ vertices of $G_{\mathcal{F}}$ outside of $H$.
Hence, there are at most $|\mathcal{F}|-|ess_{\mathcal{F}}(X)|+|st(U)|$ vertices in $H$. 
Since $|\mathcal{F}|=|ess_{\mathcal{F}}(X)|+r$, we have
$|below(U)| =|V(H)| \leqslant |\mathcal{F}|-|ess_{\mathcal{F}}(X)|+|st(U)| \leqslant |st(U)|+r.$ \hfill  $\square$ 
\end{proof}

\subsection{Structure of well-graded families with additionality 1}\label{strpartcubewithF-ess=1}

In Theorem \ref{characterizeextvc=1withequalities,right=impliesleft=}, 
we give some characterizations for well-graded families with additionality 1.
A characterization of extremal systems of VC-dimension at most $1$ was given in  \cite{MeszarosRonyaismallvc}(Proposition 2).
In Theorem \ref{characterizeextvc=1withequalities,right=impliesleft=}, 
we also give some characterizations for well-graded families of VC-dimension at most 1 and meanwhile, extend the mentioned result of \cite{MeszarosRonyaismallvc} by giving a few new characterizations for extremal systems of VC-dimension at most $1$. 

\begin{dfn}\label{defdownwardrootedtree}
We call a directed rooted tree \textit{uniformly directed} if the direction of every edge is from that vertex of the edge which is closer to the root to the other one.
\end{dfn}
\begin{dfn}\label{defonewaypath}
By a \textit{one-way path} we mean a directed path between two ending vertices $A$ and $B$ such that its edge directions are all towards one of its ending vertices, say $B$.
In other words, it is a path which is also a uniformly directed rooted tree and its root is one of its ending vertices.
\end{dfn}

We again remind that when we say that a directed graph is connected, we mean that its underlying undirected graph is connected.

\begin{thm}\label{characterizeextvc=1withequalities,right=impliesleft=}
Let $(X,\mathcal{F})$ be a set system. 
Then the following are equivalent.
\begin{enumerate}
\item{The system $(X,\mathcal{F})$ is a well-graded family with VC-dimension at most 1.}\label{partcubevcatmost1}
\item{The system $(X,\mathcal{F})$ is extremal with VC-dimension at most 1.}\label{extofvc1}
\item{The graph $G_{\mathcal{F}}$ is connected and the inequalities of conclusion of Proposition \ref{FessEGFinequalities} turn to equality.}\label{connectedandineqsareequ}
\item{The system $(X,\mathcal{F})$ is well-graded and right inequality of conclusion of Proposition \ref{FessEGFinequalities} turns to equality.}\label{pathdenseandrightinequisequ} 
\item{The system $(X,\mathcal{F})$ is well-graded with additionality 1
	 (or equivalently, the system is well-graded and the left inequality of conclusion of Proposition \ref{FessEGFinequalities} turns to equality).}\label{pathdenseandleftinequisequ} 
\item{The graph $G_{\mathcal{F}}$ is a tree and all labels on its edges are different.}\label{GFistreealllabelsaredifferent}
\item{The graph $G_{\mathcal{F}}$ is a tree with $|ess_{\mathcal{F}}(X)|+1$ vertices and $|ess_{\mathcal{F}}(X)|$ edges where labels of edges are elements of $ess_{\mathcal{F}}(X)$ and each one is used once.}\label{GFtree|ess|+1vertices}
\item{For every $A \in \mathcal{F}$, 
$G_{\mathcal{F}'}$ is a uniformly directed rooted tree with root $v_{\emptyset}$ 
and distinct labels on the edges where $\mathcal{F}'=flip(\mathcal{F}, A \rightarrow \emptyset)$.}\label{downwardtree}
\item{Same as \ref{downwardtree} with "every" replaced by "some".}\label{downwardtreeforsome}
\end{enumerate}
\end{thm}
\begin{proof}
\ref{extofvc1} $\Rightarrow$ \ref{partcubevcatmost1}) This holds since by Remark \ref{Grecohamdisequalsgraphdis}, 
extremal set systems are well-graded.

\ref{extofvc1} $\Leftrightarrow$ \ref{GFistreealllabelsaredifferent})
This is Proposition 2 of \cite{MeszarosRonyaismallvc}.

\ref{GFtree|ess|+1vertices} $\Rightarrow$ \ref{connectedandineqsareequ}, \ref{GFtree|ess|+1vertices} $\Rightarrow$ \ref{pathdenseandrightinequisequ}, \ref{GFtree|ess|+1vertices} $\Rightarrow$ \ref{pathdenseandleftinequisequ})
and
\ref{downwardtreeforsome} $\Rightarrow$ \ref{GFistreealllabelsaredifferent})
Easy.

\ref{GFtree|ess|+1vertices} $\Rightarrow$ \ref{GFistreealllabelsaredifferent} and \ref{downwardtree} $\Rightarrow$ \ref{downwardtreeforsome}) Trivial.

\ref{GFistreealllabelsaredifferent} $\Rightarrow$ \ref{GFtree|ess|+1vertices})
It is easy to see that $(X,\mathcal{F})$ is a well-graded family.
So by using Remark \ref{|ess(X)|numberofdistinctlabels},
every $x \in ess_{\mathcal{F}}(X)$ appears as the label of some edge of $G_{\mathcal{F}}$. Since labels of different edges are different, every $x \in ess_{\mathcal{F}}(X)$ appears as the label of exactly one edge. 
So there are exactly $|ess_{\mathcal{F}}(X)|$ edges. It follows that 
$G_{\mathcal{F}}$ is a tree with $|ess_{\mathcal{F}}(X)|+1$ vertices and $|ess_{\mathcal{F}}(X)|$ edges and different edges have different labels.

\vspace{1mm}

\ref{connectedandineqsareequ} $\Rightarrow$ \ref{GFistreealllabelsaredifferent})
Since $|\mathcal{F}|=|E(G_{\mathcal{F}})|+1$ and 
$|V(G_{\mathcal{F}})|=|\mathcal{F}|$ we have $|V(G_{\mathcal{F}})|=|E(G_{\mathcal{F}})|+1$. So since $G_{\mathcal{F}}$ is connected, $G_{\mathcal{F}}$ is a tree.
On the other hand, by Remark \ref{|ess(X)|numberofdistinctlabels} 
every $x \in ess_{\mathcal{F}}(X)$ appears as the label of at least one edge in $E(G_{\mathcal{F}})$.
Now since by assumption $|ess_{\mathcal{F}}(X)|=|E(G_{\mathcal{F}})|$, every edge in $E(G_{\mathcal{F}})$ has a different label.

\vspace{1mm}

\ref{pathdenseandrightinequisequ} $\Rightarrow$ 
\ref{GFistreealllabelsaredifferent})
We recall that the one-inclusion graphs of 
well-graded families are connected.
By assumption 
$|\mathcal{F}|=|E(G_{\mathcal{F}})|+1$. 
Note that $|\mathcal{F}|=|V(G_{\mathcal{F}})|$.
So the one-inclusion graph $G_{\mathcal{F}}$ is a tree. We prove that labels on edges are different. Assume not and there are two edges with the same labels. So there exist a path $P$ in $G_{\mathcal{F}}$ such that contains two edges with the same labels. Let $v_1$ and $v_2$ be the two ending vertices of $P$.
Since there are edges with the same labels in the path, $|st(v_1) \triangle st(v_2)|$ is strictly less than the number of edges of the path.
On the other hand, $P$ is the only path between $v_1$ and $v_2$ since $G_{\mathcal{F}}$ is a tree. It follows that 
$d_{G_{\mathcal{F}}}(v_1,v_2)$ is equal to the number of edges of $P$.
But now $|st(v_1) \triangle st(v_2)| < d_{G_{\mathcal{F}}}(v_1,v_2)$ which is a contradiction with the assumption of well-gradedness of $(X,\mathcal{F})$.
So labels are all different.

\vspace{1mm}

\ref{pathdenseandleftinequisequ} $\Rightarrow$ \ref{downwardtree})
Let $A \in \mathcal{F}$ be arbitrary and $\mathcal{F}':=flip(\mathcal{F}, A \rightarrow \emptyset)$.
One can see that $(X,\mathcal{F}')$ is also a well-graded family and similar to $\mathcal{F}$, 
we have $|ess_{\mathcal{F}'}(X)|+1 = |\mathcal{F}'|$. So additionality of $(X,\mathcal{F}')$ is 1.

\vspace{1mm}

\textit{Claim 1.} For every $B \in \mathcal{F}'$ we have 
$|\{C \in \mathcal{F}': C \subseteq B\}|=|B|+1$.

\vspace{1mm}

\textit{Proof of Claim 1.}
Fix some $B \in \mathcal{F}'$ and let $U:=\{v_B\}$. 
So $st(U)=B$ where we remind that $st(U)$ was defined in Definition \ref{dfnlabeless}.
One can see that 
$\{v_C:C \in \mathcal{F}', C \subseteq B\}=below(U)$.
Now by applying Lemma \ref{sizeofbelow} for the system $(X,\mathcal{F}')$ and the defined $U$ and since the parameter $r$ appeared in that lemma in here is $1$, we have 
$|\{C \in \mathcal{F}': C \subseteq B\}|=|\{v_C:C \in \mathcal{F}', C \subseteq B\}|=|below(U)|=|st(U)|+1=|B|+1.$
\hfill \textit{Claim 1} $\square$

\vspace{1mm}

\textit{Claim 2.}
Let $P_1, P_2$ be two one-way path's in $G_{\mathcal{F}'}$ started from $v_{\emptyset}$ with vertices $V(P_1)=\{v_{\emptyset},w_1,\ldots,w_r\}$ and $V(P_1)=\{v_{\emptyset},z_1,\ldots,z_s\}$. Then $P_1$ and $P_2$ can have common vertices only in an initial part of themselves.
More precisely, there is some $t \leqslant min(r,s)$ such that $w_i=z_i$ for every $1 \leqslant i \leqslant t$ and these vertices and also $v_{\emptyset}$ are the only common vertices of $P_1$ and $P_2$.

\vspace{1mm}

\textit{Proof of claim 2.}
Assume that $v_C \in V(P_1) \cap V(P_2)$ for some $C \in \mathcal{F}'$.
It is enough to show that the set of vertices before $v_C$ (in the order of the one-way path) in $P_1$ is the same as the one in $P_2$.
Since $P_1$ and $P_2$ are directed path's starting from $v_{\emptyset}$ and because of the way that $G_{\mathcal{F}'}$ is defined, the number of vertices of $G_{\mathcal{F}'}$ before $v_C$ in each of $P_1$ and $P_2$ is exactly $|C|$. Note that every vertex before $v_C$ in each of $P_1$ and $P_2$ corresponds to a proper subset of $C$ in $\mathcal{F}'$. But now by using Claim 1, the number of elements of $\mathcal{F}'$ which are proper subsets of $C$ is $|C|$. So vertices before $v_C$ in both $P_1$ and $P_2$ should be the same. 
\hfill \textit{Claim 2} $\square$

\vspace{1mm}

Let $\{A_1,\ldots,A_r\}$ be a set of maximal elements of $\mathcal{F}'$ (i.e. for each $j\leqslant r$, $A_j$ is not a proper subset of any other element of $\mathcal{F}'$) and assume that it covers $ess_{\mathcal{F}'}(X)$ which means that $ess_{\mathcal{F}'}(X) \subseteq \bigcup_{i\leqslant r}A_i$. 
Note that such a set exists since for example the set of all maximal elements of $\mathcal{F}'$ covers $ess_{\mathcal{F}'}(X)$.
Also note that by Remark \ref{infliptoemtryeverymemberissubsetofess}, for every $A_i$ we have 
$A_i \subseteq ess_{\mathcal{F}}(X)=ess_{\mathcal{F}'}(X)$ which follows that $ess_{\mathcal{F}'}(X) = \bigcup_{i\leqslant r}A_i$. 
Since $\mathcal{F}'$ is a well-graded family, for each $i \leqslant r$ there is a one-way path in $G_{\mathcal{F}'}$ denoted by $P_i$ directed from $v_{\emptyset}$ to $v_{A_i}$. 
Since $A_i$'s are among maximal elements of $\mathcal{F}'$, none of $P_i$'s is contained in another one.
Let $H$ be the subgraph of $G_{\mathcal{F}'}$ consisting of the union of all $P_i$'s. Since $v_{\emptyset} \in V(H)$, $H$ is connected and moreover, by Claim 2 is a tree. In fact, $H$ is a uniformly directed rooted tree with root $v_{\emptyset}$.
Now we show that $H=G_{\mathcal{F}'}$.

It is easy to see that every element of $\bigcup_{i\leqslant r} A_i$ appears as the label of some edge of $H$. It follows that $H$ has at least $|\bigcup_{i\leqslant r} A_i|$ edges. Therefore, since $H$ is a tree it has at least $|\bigcup_{i\leqslant r} A_i|+1$ vertices.
Also since 
$ess_{\mathcal{F}'}(X)=\bigcup_{i\leqslant r}A_i$,
we have $|\bigcup_{i\leqslant r} A_i|+1=|ess_{\mathcal{F}'}(X)|+1$.
Thus $|V(H)|\geqslant |\bigcup_{i\leqslant r} A_i|+1=|ess_{\mathcal{F}'}(X)|+1$. 
One can see that $ess_{\mathcal{F}}(X)=ess_{\mathcal{F}'}(X)$ and so now by using the assumption $|\mathcal{F}|=|ess_{\mathcal{F}}(X)|+1$ we have $|V(G_{\mathcal{F}'})|=|\mathcal{F}'|=|\mathcal{F}|=|ess_{\mathcal{F}}(X)|+1=|ess_{\mathcal{F}'}(X)|+1 \leqslant |V(H)|$.
On the other hand, 
$V(H) \subseteq V(G_{\mathcal{F}'})$.
Hence, $V(H)=V(G_{\mathcal{F}'})$.
Now we claim that $G_{\mathcal{F}'}$ does not have any more edges than $H$. 
Assume for contradiction that for some $v_1,v_2 \in V(H)$, $v_1v_2$ is an edge (with direction from $v_1$ to $v_2$) in $E(G_{\mathcal{F}'}) \setminus E(H)$. One can see that there is no $i \leqslant r$ such that $v_1$ and $v_2$ both belong to $P_i$ since there is no edge in $E(G_{\mathcal{F}'})$ between any two non-consecutive vertices of any $P_i$. So there are distinct $m,n \leqslant r$ such that $v_1 \in V(P_m) \setminus V(P_n)$ and $v_2 \in V(P_n) \setminus V(P_m)$. Let $P_3$ be the one-way path formed by the union of the initial part of $P_m$ from $v_{\emptyset}$ to $v_1$ and the edge $v_1v_2$.
Also let $P_4$ be the initial part of $P_n$ from $v_{\emptyset}$ to $v_2$. Now one can see that the common vertices of two one-way path's $P_3$ and $P_4$ do not form an initial part of them which is a contradiction with Claim 2.
Hence, $E(G_{\mathcal{F}'}) = E(H)$. 
So we have shown that $H=G_{\mathcal{F}'}$.
Thus, $G_{\mathcal{F}'}$ is a uniformly directed rooted tree with root $v_{\emptyset}$. 
Therefore, we have $|\mathcal{F}'|=|V(G_{\mathcal{F}'})|=|E(G_{\mathcal{F}'})|+1$ and moreover, since $\mathcal{F}'$ is a well-graded family, by (\ref{pathdenseandrightinequisequ} $\Rightarrow$ 
\ref{GFistreealllabelsaredifferent}) of this theorem one sees that the labels on edges of $G_{\mathcal{F}'}$ are distinct.

\vspace{1mm}

\ref{partcubevcatmost1} $\Rightarrow$ \ref{GFistreealllabelsaredifferent})
The proof is similar to the proof of one direction of Proposition 2 of \cite{MeszarosRonyaismallvc} but with some modifications.
First we show that all labels of edges of $G_{\mathcal{F}}$ are different.
Assume for contradiction that $a$ is the label of two different edges $v_{A}v_{B}$ and $v_{C}v_{D}$ (with directions from $v_{A}$ to $v_{B}$ and from $v_{C}$ to $v_{D}$) for some $A,B,C,D \in \mathcal{F}$. So $B=A\cup \{a\}$, $D=C \cup \{a\}$ and $A$ and $C$ do not contain $a$. Also one can see that $A,B,C$ and $D$ are all distinct. Since $a \in B \cap D$, there is some $b$ different from $a$ in $B \triangle D$. Now one can verify that the set $\{a,b\}$ is shattered by four sets $A,B,C$ and $D$ which is a contradiction with the assumption that VC-dimension of $\mathcal{F}$ is at most 1. So the labels of edges of $G_{\mathcal{F}}$ are different. It follows that there is no cycle in $G_{\mathcal{F}}$ since in each cycle, every label repeats with an even number of times.
Since $(X,\mathcal{F})$ is a well-graded family, $G_{\mathcal{F}}$ is connected.
So $G_{\mathcal{F}}$ is a tree and labels of its edges are different.
\hfill $\square$
\end{proof}

\subsection{Structure of well-graded families with additionality 2}\label{strpartcubewithF-ess=2}
In this subsection, we characterize well-graded families with additionality 2 by using the one-inclusion graphs.
We first give an example of such systems. 

\begin{exa}\label{exampleofadditionality2}
Let $X:=\{a,b,c,d,e,f,g,h,i,j,k,l\}$ and 
$$\mathcal{F}:=\{a,ab,abc,ac,ace,acef,aceg,ah,abcd,abi,abj,abjk,abjl\}.$$
Then $(X,\mathcal{F})$ is a set system where any member of $\mathcal{F}$, for example $abc$, is seen as a subset of $X$, in this example $\{a,b,c\}$.
The set system $(X,\mathcal{F})$ is a well-graded family.
Also $ess_{\mathcal{F}}(X)=\{b,c,d,e,f,g,h,i,j,k,l\}$ and so $|\mathcal{F}|=|ess_{\mathcal{F}}(X)|+2$.
Figure 1 illustrates the 
one-inclusion graph of the system $(X,\mathcal{F})$.
\end{exa}

\begin{center}

\begin{tikzpicture}
\usetikzlibrary{arrows}
\tikzset{
    %vertex/.style={circle,draw,minimum size=1.5em},
    edge/.style={->,> = latex'}
}

% vertices

\vertex (a) at (0,0) [label=above:$v_{a}$] {};
\vertex (ac) at (-1,1) [label=above:$v_{ac}$] {};
\vertex (ab) at (1,1) [label=below:$v_{ab}$] {};
\vertex (abc) at (0,2) [label=left:$v_{abc}$] {};
\vertex (ace) at (-2,1) [label=below:$v_{ace}$] {};
\vertex (aceg) at (-3.25,0.5) [label=below:$v_{aceg}$] {};
\vertex (acef) at (-3.25,1.5) [label=below:$v_{acef}$] {};
\vertex (abcd) at (0,2.75) [label=above:$v_{abcd}$] {};
\vertex (ah) at (0,-0.75) [label=below:$v_{ah}$] {};
\vertex (abi) at (2.25,1.5) [label=above:$v_{abi}$] {};
\vertex (abj) at (2.25,0.5) [label=below:$v_{abj}$] {};
\vertex (abjk) at (3.5,1) [label=above:$v_{abjk}$] {};
\vertex (abjl) at (3.5,0) [label=below:$v_{abjl}$] {};

%edges

\draw[edge] (a) -- (ac) node[midway, above] {$c$};
\draw[edge] (a) -- (ab) node[midway, right] {$b$};
\draw[edge] (ac) -- (abc) node[midway, right] {$b$};
\draw[edge] (ab) -- (abc) node[midway, right] {$c$};
\draw[edge] (ac) -- (ace) node[midway, below] {$e$};
\draw[edge] (ace) -- (aceg) node[midway, below] {$g$};
\draw[edge] (ace) -- (acef) node[midway, above] {$f$};
\draw[edge] (abc) -- (abcd) node[midway, right] {$d$};
\draw[edge] (a) -- (ah) node[midway, right] {$h$};
\draw[edge] (ab) -- (abi) node[midway, above] {$i$};
\draw[edge] (ab) -- (abj) node[midway, below] {$j$};
\draw[edge] (abj) -- (abjk) node[midway, above] {$k$};
\draw[edge] (abj) -- (abjl) node[midway, below] {$l$};
\end{tikzpicture}

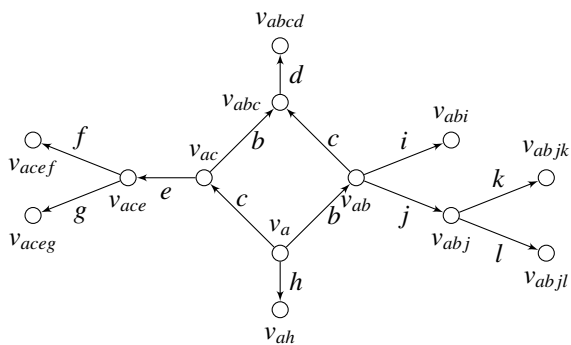
Figure 1: The one-inclusion graph of the set system of Example \ref{exampleofadditionality2}.

\end{center}

We fix some notations.
By a \textit{parallel-directed-labelled $C_4$ cycle} we mean a labelled $C_4$ cycle graph $a_1a_2a_3a_4$ 
where the edges $a_1a_2$ and $a_4a_3$ in cycle have the same labels and also parallel directions (i.e. they are directed either as $a_1$ to $a_2$ and $a_4$ to $a_3$ or directed as $a_2$ to $a_1$ and $a_3$ to $a_4$) and similarly, $a_2a_3$ and $a_4a_1$ have the same labels (but different from labels of the edges $a_1a_2$ and $a_4a_3$) and parallel directions.

\begin{dfn}\label{defgraphstructureofF=essX+2}
We call a labelled directed graph 
$G$ a \textit{semitree} if 
$G$ consists of a parallel-directed-labelled $C_4$ cycle $a_1a_2a_3a_4$ such that four (possibly empty) trees are attached to each of the vertices $a_1,a_2,a_3$ and $a_4$ and moreover, labels of edges outside of $C_4$ cycle are different from each other and from the labels of the $C_4$ cycle.
\end{dfn}

As an example of a semitree one can see the one-inclusion graph illustrated in Figure 1.
It is easy to see that if $G$ is a semitree and $ab$ is an edge and $o \in V(G)$, then $d_G(o,a)$ and $d_G(o,b)$ differ exactly by one unit.
\begin{dfn}
We call a semitree $G$ \textit{uniformly directed} if either $G$ is a parallel-directed-labelled $C_4$ cycle or there is a 
vertex $o \in V(G)$ 
such that for every edge $ab$ of $G$, the direction of the edge 
is from $a$ to $b$ if and only if $d_G(o,b)> d_G(o,a)$. 
\end{dfn}

The following result characterizes well-graded families with additionality 2 via the one-inclusion graphs.

\begin{thm}\label{structureofpathdensewithF=ess+2}
Assume that $(X,\mathcal{F})$ is a set system. 
Then the following are equivalent.
\begin{enumerate}
\item{The system $(X,\mathcal{F})$
is a well-graded family with additionality 2. 
}\label{pathdensewithF=ess+2}
\item{For every $D \in \mathcal{F}$, the graph $G_{\mathcal{F}'}$
is a uniformly directed semitree  
where $\mathcal{F}':=flip(\mathcal{F}, D \rightarrow \emptyset)$.}\label{1-incgraphofeveryflipissemitree}
\item{Same as \ref{1-incgraphofeveryflipissemitree} with "every" replaced by "some".}\label{1-incgraphofsomeflipissemitree}
\item{The graph $G_{\mathcal{F}}$ is a semitree graph.}\label{1-incgraphissemitree}
\end{enumerate}
\end{thm}
\begin{proof}
($\ref{pathdensewithF=ess+2} \Rightarrow \ref{1-incgraphofeveryflipissemitree}$) 
Fix some $D \in \mathcal{F}$. For simplifying the notations, 
we use the notation $G$ instead of $G_{\mathcal{F}'}$, the one-inclusion graph associated to $\mathcal{F}':=flip(\mathcal{F}, D \rightarrow \emptyset)$.
One can see that $(X,\mathcal{F}')$ is also a well-graded family and by Remark \ref{infliptoemtryeverymemberissubsetofess}, its additionality is 2.
So $G$ is a connected graph where we remind that when we say that a directed graph is connected, we mean that its underlying undirected graph is connected.
Also note that $\emptyset \in \mathcal{F}'$.

\vspace{1mm}

\textit{Claim 1.} There is a unique vertex in $G$ with incoming degree $2$. Also the incoming degrees of other vertices are at most 1.

\textit{Proof of claim 1.}
We first show that there is at least one vertex with incoming degree at least $2$. Assume for contradiction that every vertex of $G$ has incoming degree at most $1$. 
Note that $v_{\emptyset}$ is a vertex in $G$ with incoming degree zero. 
Now it is not hard to see that $G$ does not have any cycle
since if there is a cycle $C$, then the vertex $v_0 \in V(C)$ with $|st(v_0)|=max(\{|st(v)|:v \in V(C)\})$ has incoming degree 2 in $C$ 
which is in contradiction with our assumption.
So since $G$ is connected, it is a tree. Thus, $|\mathcal{F}'|=|V(G)|=|E(G)|+1$. Now by Theorem \ref{characterizeextvc=1withequalities,right=impliesleft=}(\ref{pathdenseandrightinequisequ} $\Rightarrow$ \ref{connectedandineqsareequ}) we have 
$|\mathcal{F}'|=|ess_{\mathcal{F}'}(X)|+1$ which is a contradiction with the fact that $(X,\mathcal{F}')$ has additionality 2.
So there are vertices with incoming degrees at least $2$.

Now we show that incoming degree of every vertex in $G$ is at most $2$. 
Assume for contradiction that there are $a,b,c,d \in V(G)$ such that $ba,ca,da$ are three edges with directions incoming to the vertex $a$.
One can see that $below(\{b\}) \subseteq below(\{a\})$ while $a,c,d \not \in below(\{b\})$ and $a,c,d \in below(\{a\})$ where these notations were defined in Definition \ref{dfnlabeless}.
So we have $|below(\{a\})| \geqslant |below(\{b\})|+3$.
Since $(X,\mathcal{F}')$ is a well-graded family, there is a one-way path in $G$ from $v_{\emptyset}$ to $b$.
It is not hard to see that the number of vertices on this path is $|st(b)|+1$
and all of these vertices belong to $below(\{b\})$. So $|below(\{b\})| \geqslant |st(b)|+1$.
Also it is easily seen that $|st(b)|+1=|st(a)|=|st(\{a\})|$.
Hence, $|below(\{b\})| \geqslant |st(\{a\})|$.
Combining the above inequalities, we get
$|below(\{a\})| \geqslant |st(\{a\})|+3$ which is a contradiction with Lemma \ref{sizeofbelow} and the fact that additionality is $2$. So the incoming degree of every vertex of $G$ is at most $2$.

So far by combining the above paragraphs we have shown that there are vertices with incoming degrees exactly $2$. Now we show that such a vertex is unique.
Assume for contradiction that $a_0,a_1$ are two vertices with incoming degrees $2$ and $b_0a_0$, $c_0a_0$ and also $b_1a_1$ and $c_1a_1$ are the directed edges incoming to $a_0$ and $a_1$ respectively. 
Note that the sets $\{b_0,c_0\}$ and $\{b_1,c_1\}$ can intersect.
We divide the situation to the following two cases.

\textit{Case I.} In this case we assume that none of $st(a_0)$ and $st(a_1)$ is subset of the other one (or equivalently $a_0 \not \in below(\{a_1\})$ and $a_1 \not \in below(\{a_0\})$).
One can see that $st(\{b_0,c_0\})=st(\{a_0\})$ and $st(\{b_1,c_1\})=st(\{a_1\})$. So $st(\{b_0,c_0,b_1,c_1\})=st(\{a_0,a_1\})$.
By the assumption that none of $st(a_0)$ and $st(a_1)$ is a subset of the other one we have $a_0,a_1 \in below(\{a_0,a_1\}) \setminus below(\{b_0,c_0,b_1,c_1\})$.
Also clearly $below(\{b_0,c_0,b_1,c_1\}) \subseteq below(\{a_0,a_1\})$.
Now by combining these facts with the left inequality of Lemma \ref{sizeofbelow} we have
$$|below(\{a_0,a_1\})| \geqslant |below(\{b_0,c_0,b_1,c_1\})|+2
\geqslant |st(\{b_0,c_0,b_1,c_1\})|+1+2=|st(\{a_0,a_1\})|+3$$ which is a contradiction with Lemma \ref{sizeofbelow} and the fact that additionality is $2$.

\vspace{1mm}

\textit{Case II.} In this case we assume that one of $st(a_0)$ or $st(a_1)$ is a subset of the other one.
Without loss of generality, we may assume that $st(a_1) \subseteq st(a_0)$. 
Because of the well-gradedness of the system $(X,\mathcal{F}')$, there is a directed one-way path from $a_1$ to $a_0$. Since $b_0a_0$ and $c_0a_0$ are the only incoming edges to $a_0$, this path should pass through exactly one of $b_0$ or $c_0$. Hence, either $st(a_1) \subseteq st(b_0)$ or $st(a_1) \subseteq st(c_0)$.
We may assume that $st(a_1) \subseteq st(b_0)$. 
Again, because of the well-gradedness of the system $(X,\mathcal{F}')$, there is a one-way path $P$ which starts from $v_{\emptyset}$ and ends in $b_0$ and passes through $a_1$.
Since $b_1a_1$ and $c_1a_1$ are the only incoming edges to $a_1$, the path $P$ must also pass either through $b_1$ or $c_1$ and not both.
Without loss of generality, assume that $P$ passes through $b_1$ and not $c_1$.
One can see that $|V(P)|=|st(b_0)|+1$ and also $V(P) \cup \{c_1\} \subseteq below(\{b_0\})$. So $|below(\{b_0\})| \geqslant |V(P)|+1$.
Hence, $|below(\{b_0\})| \geqslant |st(\{b_0\})|+2$.
Note that $a_0,c_0 \in below(\{a_0\}) \setminus below(\{b_0\})$ and $below(\{b_0\}) \subseteq below(\{a_0\})$.
Therefore, it is not difficult to verify that we have
$$|below(\{a_0\})| \geqslant |below(\{b_0\})|+2 \geqslant |st(\{b_0\})|+2+2
=|st(\{a_0\})|-1+4=|st(\{a_0\})|+3$$ which is a contradiction with Lemma \ref{sizeofbelow} and the fact that additionality is $2$.

So in both cases we got contradiction. Therefore, there is a unique vertex with incoming degree $2$. Also all other vertices have incoming degrees at most one.
\hfill \textit{Claim 1} $\square$

\vspace{1mm}

Denote the unique vertex with incoming degree $2$ by $a_0$
and assume that $b_0a_0$ and $c_0a_0$ are the only incoming edges to $a_0$ for some vertices $b_0$ and $c_0$.
Obviously $b_0,c_0 \in below(\{a_0\})$.
Because of the well-gradedness of the system $(X,\mathcal{F}')$, there is a one-way path $P$ which starts from $v_{\emptyset}$ and ends in $a_0$. It is easy to see that $P$ has $|st({a_0})|+1$ vertices.
Since $b_0$ and $c_0$ are the only incoming edges to $a_0$, $P$ must pass through exactly one of them. 
Without loss of generality, we assume that $P$ passes through $b_0$. By right inequality of Lemma \ref{sizeofbelow} and the fact that the additionality is $2$, we have $|below(\{a_0\})|\leqslant |st(\{a_0\})|+2=|V(P)|+1$.
Also it is obvious that $V(P) \subseteq below(\{a_0\})$ and $c_0 \in below(\{a_0\}) \setminus V(P)$. So combining these facts, we have $below(\{a_0\})=V(P) \cup \{c_0\}$. 
Having this, it is not difficult to see that there is a unique vertex $v \in below(\{a_0\})$ with $|st(\{v\})|=|st(\{a_0\})|-2$ and it is the vertex before $b_0$ in the path $P$ (with respect to the one-way direction on $P$). Denote this vertex by $d_0$.
Again, by well-gradedness of $(X,\mathcal{F}')$, there is a one-way path $P'$ from $v_{\emptyset}$ to $c_0$. Let $y$ be the vertex before $c_0$ in $P'$.
Obviously $V(P') \subseteq below(\{a_0\})$ and $y \in below(\{a_0\})$. 
Also we have $|st(\{y\})|=|st(\{c_0\})|-1=|st(\{a_0\})|-2$. Since $d_0$ was the unique vertex with the mentioned property, we have $y=d_0$. Therefore, $d_0b_0$ and $d_0c_0$ are two directed edges with directions from $d_0$ to $b_0$ and $d_0$ to $c_0$ respectively.
We remind that the directions of $c_0a_0$ and $b_0a_0$ are from $b_0$ and $c_0$ to $a_0$.
Now $d_0b_0a_0c_0$ forms a $C_4$ cycle (which we denote it by $C_0$).
One can see that $lab(d_0b_0)=lab(c_0a_0)$ and $lab(d_0c_0)=lab(b_0a_0)$.
So $C_0$ is a parallel-directed-labelled $C_4$ cycle.
Also if $d_0 \not =v_{\emptyset}$, the initial part of the above defined path $P$ from $v_{\emptyset}$ to $d_0$ would be a one-way path between those two vertices. 

\vspace{1mm}

\textit{Claim 2.} $C_0$ is the only cycle in $G$.

\textit{Proof of claim 2.} 
Assume that $C$ is a cycle in $G$. We show that $C$ is the same as $C_0$. Assume that $a$ is a vertex of $C$ such that the set $st(a)$ is maximal (in the sense of set-inclusion order) in the $\{st(v):v \in V(C)\}$. Because of such a maximality mentioned for $a$, 
it has incoming degree $2$ and outgoing degree $0$ in $C$. Let $ba$ and $ca$ be the incoming edges to $a$ in $C$. 
Since $a_0$ 
has incoming degree $2$ and by Claim 1, the vertex with incoming degree 2 in $G$ 
is unique and other vertices of $G$ have incoming degrees less than or equal to $1$, we have $a=a_0$.
We remind that $b_0a_0$ and $c_0a_0$ are two incoming edges to $a_0$. So $ba$ and $ca$ are exactly these two incoming edges to $a_0 (=a)$ in $G$. 
Without loss of generality, we may assume that $b=b_0$ and $c=c_0$. 
Assume that $e$ and $k$ are vertices of $C$ (different from $a,b$ and $c$ and not necessarily distinct from each other) such that $eb$ and $kc$ are edges of $C$. 
In order to show that two cycles $C$ and $C_0$ are the same, it would be enough to show that $d_0=e=k$. Assume for contradiction that this does not hold and without loss of generality, assume that $d_0 \not = e$.
Note that in this case the direction of the edge $eb$ is from $b$ to $e$ since otherwise $b$ would have incoming degree $2$ (since direction of $d_0b$ is also incoming to $b$) which contradicts 
Claim 1 and uniqueness of $a (=a_0)$ as the only vertex with incoming degree greater than 1.
We first show that in this situation $d_0$ is different from $k$ too.
If $d_0=k$, then again by using the uniqueness of $a$ as the only vertex with incoming degree greater than 1 and also the mentioned fact that direction of the edge $eb$ is from $b$ to $e$, one sees that the arc of $C$ which starts from $b$, passes through $e$ and ends in $k (=d_0)$ is a one-way path directed from $b$ to $d_0$.
So we have $|st(b)|<|st(e)|<|st(d_0)|$. But since $d_0b$ is also an edge in $G$ with direction from $d_0$ to $b$, we have $|st(d_0)| < |st(b)|$ which is a contradiction. So $d_0$ is different from $k$ too.
A similar argument as before, this time for the edge $ck$ instead of $eb$, shows that the direction of the edge $ck$ must be from $c$ to $k$.
Now we start from the vertex $e$ and move on edges of $C$ according to the directions on the edges until when it is impossible to move anymore which means that we have arrived to a vertex $t$ such that the next edge of $C$ adjacent to it is also incoming to $t$. So in this case $t$ is a vertex with incoming degree $2$. Hence, again by Claim 1 and uniqueness of $a_0$ as the only vertex with incoming degree greater than $1$, we have $t=a \ (=a_0)$. 
But since direction of the edge $ck$ is from $c$ to $k$, the above movements on the edges of $C$
must have stopped before arriving to the vertex $c$. Thus, $t$ must be different from $a$ which is a contradiction. 
Therefore, $d_0=e$ and also by a similar argument, we have $d_0=k$.
It follows that $C=C_0$ and there is only one cycle in $G$ which is a $C_4$ cycle.
\hfill \textit{Claim 2} $\square$

\vspace{1mm}

Since the system $(X,\mathcal{F}')$ is a well-graded family, $G$ is connected.
Also since by Claim 2, the graph $G$ contains exactly one cycle which is the above mentioned parallel-directed-labelled $C_4$ cycle $d_0b_0a_0c_0$, there are four (possibly empty) trees attached to each of the vertices $d_0,b_0,a_0$ and $c_0$.
Again, because of well-gradedness of $(X,\mathcal{F}')$, it is not very difficult to verify that the labels of edges outside of the $C_4$ cycle are different from each other and from the labels of the $C_4$ cycle. Thus, $G$ is a semitree graph. 
We remind that $G$ contains $v_{\emptyset}$. Let $e$ be any arbitrary edge of $G$ connecting two vertices $v_A$ and $v_B$ for some $A,B \in \mathcal{F}'$. 
So $|A \triangle B|=1$.
Because of the well-gradedness of $(X,\mathcal{F}')$, 
we have $d_G(v_{\emptyset},v_B)=|B|$ and $d_G(v_{\emptyset},v_A)=|A|$.
So the direction of $e$ is from $v_A$ to $v_B$ if and only if $A \subseteq B$ and $|B \setminus A|=1$ if and only if $d_G(v_{\emptyset},v_B) > d_G(v_{\emptyset},v_A)$.
It follows that $G$ is a uniformly directed semitree.

($\ref{1-incgraphofeveryflipissemitree} \Rightarrow \ref{1-incgraphofsomeflipissemitree}$) Obvious.

($\ref{1-incgraphofsomeflipissemitree} \Rightarrow \ref{1-incgraphissemitree}$) 
By assumption, there is some $D \in \mathcal{F}$ such that $G_{\mathcal{F}'}$ is a uniformly directed semitree where $\mathcal{F}':=flip(\mathcal{F}, D \rightarrow \emptyset)$.
First note that $\mathcal{F}=flip(\mathcal{F}', \emptyset \rightarrow D)$.
We remind that the bit-flip and flip operations on a set system do not change the edges of the one-inclusion graph and their labels and only might change the directions of some of the edges.
Moreover, one can see that after flip operations any parallel-directed-labelled $C_4$ cycle also turns to a parallel-directed-labelled $C_4$ cycle.
So since the one-inclusion graph of $\mathcal{F}'$ is a semitree graph and $\mathcal{F}$ is obtained from $\mathcal{F}'$ by flip operations, then $G_{\mathcal{F}}$ is also a semitree graph.

($\ref{1-incgraphissemitree} \Rightarrow \ref{pathdensewithF=ess+2}$)
Assume that $G_{\mathcal{F}}$ is a semitree graph.
Therefore, by definition 
labels of the edges outside of the $C_4$ cycle of $G_{\mathcal{F}}$ are distinct from each other and from the labels of edges of the $C_4$ cycle. Also every pair of non-adjacent edges of the $C_4$ cycle have the same labels but different from the labels of the other pair.
So it is not hard to verify that 
$\mathcal{F}$ is a well-graded family. 
Note that in a semitree, the number of edges and vertices are the same. Also remind that the labels of the parallel edges in $C_4$ cycle are the same. So the number of distinct labels in $G_{\mathcal{F}}$ is $|\mathcal{F}|-2$. On the other hand, by Remark $\ref{|ess(X)|numberofdistinctlabels}$ 
we know that 
$|ess_{\mathcal{F}}(X)|$ is equal to the number of distinct labels in $G_{\mathcal{F}}$. So $|\mathcal{F}|-2=|ess_{\mathcal{F}}(X)|$. It follows that the additionality of the system is equal to 2. \hfill $\square$
\end{proof}

\vspace{1mm}

The following statement determines the VC-dimension of well-graded families with additionality $2$.

\begin{rem}\label{VCdimadditionality2}
\begin{enumerate}
\item{If the VC-dimension of a set system is less than or equal to 1, then the additionality of it is less than or equal to 1 too. So every set system with additionality 2 has VC-dimension at least 2.}\label{VCdimadditionality2vclessthaneq1}

\item{The VC-dimension of every well-graded family with additionality 2 is equal to $2$.}\label{VCdimadditionality2WGvcis2}
\end{enumerate}
\end{rem}
\begin{proof}
\ref{VCdimadditionality2vclessthaneq1})
If VC-dimension of a set system $(X,\mathcal{F})$ with nonempty $\mathcal{F}$ is $0$, then $|\mathcal{F}|=1$, $ess_{\mathcal{F}}(X)=\emptyset$ and additionality is 1. So assume that $VCdim(\mathcal{F})=1$. Let $D:=ess_{\mathcal{F}}(X)$ and $\mathcal{F}':=\mathcal{F} \cap D$. 
It is easily seen that $|\mathcal{F}'|=|\mathcal{F}|$ and also $VCdim(\mathcal{F}')=1$. Now by applying Sauer-Shelah lemma (Theorem \ref{SauerShelahLemma}) on the set system $(D,\mathcal{F}')$ we have
$$|\mathcal{F}|=|\mathcal{F}'| \leqslant \sum_{i=0}^{1}\binom{|D|}{i}=|D|+1=|ess_{\mathcal{F}}(X)|+1.$$
So the additionality of the system is less than or equal to 1.

\vspace{1mm}

\ref{VCdimadditionality2WGvcis2})
Let $(X,\mathcal{F})$ be a well-graded family with additionality 2. 
So by part \ref{VCdimadditionality2vclessthaneq1}, VC-dimension of the system is at least 2.
By Theorem $\ref{structureofpathdensewithF=ess+2}(\ref{pathdensewithF=ess+2} \Rightarrow \ref{1-incgraphissemitree})$, 
the graph $G_{\mathcal{F}}$ is a semitree graph. Assume that $a$ and $b$ are the labels of the edges of the parallel-directed-labelled $C_4$ cycle of this graph. 
It worth mentioning and not hard to see that $\{a,b\}$ as a subset of $X$ is shattered by $\mathcal{F}$.
We claim that there is no three element subset of $X$ which is shattered by $\mathcal{F}$. 
The reason is that if $\{w_1,w_1,w_3\} \subseteq X$ is shattered, 
then by Remark \ref{inWGeveryelemofshatteredsetappearslabel}, each $w_i$ appears as the label of at least $4$ edges of $G_{\mathcal{F}}$. But one can see that since $G_{\mathcal{F}}$ is a semitree graph and by definition of semitree graphs this is impossible to happen.
Therefore, $VCdim(\mathcal{F})=2$.
\hfill $\square$

\end{proof}

\section{Characterization of self-and-dual well-graded families}\label{sectioncharselfanddualpartcube}

The main goal of this section is Theorem \ref{characterizationselfanddualpathdense}, where we characterize self-and-dual well-graded families.
Note that apart from its own interest, this characterization will have some applications in studying the graphs which we will discuss in the next section.
This section uses the results of previous sections in particular Section \ref{strofpartcubewithsmalladditionalities} and the structural characterizations of the systems with additionalities 1 and 2 in there. 
Similar to the previous sections, many techniques of this section work on the base of analysing the one-inclusion graphs of the set systems.
We first define some new notions and state a few lemmas about them.

\begin{dfn}\label{defgettingcloserfar}
Let $v$ be a vertex of a connected directed graph $G$ and $e$ an edge of the graph connecting two vertices $w$ and $z$ and directed from $w$ to $z$.
We say that (the direction of) $e$ is getting closer to $v$ if 
$d_G(w,v) > d_G(z,v)$. Also we say that $e$ is getting far from $v$ if $d_G(w,v) < d_G(z,v)$.
\end{dfn}

\begin{rem}\label{indirectedtreeeveryedgegetingcloseorfar}
In a directed tree, 
for any edge $e$ and vertex $v$, either $e$ is getting closer to $v$ or getting far from it.
\end{rem}

\begin{dfn}
We call a vertex $v$ of a connected directed graph $G$ a \textit{source} if every edge of $G$ is getting far from $v$. Also we call $v$ a \textit{sink} if every edge is getting closer to $v$.
\end{dfn}

\begin{rem}\label{sourcesinkindirectedtrees}
A connected directed graph $G$ has at most one source and at most one sink. 
Moreover, the only directed tree which has both source and sink is a one-way path.
\end{rem}
\begin{proof}
Let $G$ be a connected directed graph.
It is not hard to see that for any two vertices of $G$, every edge of the shortest path connecting those vertices together is getting closer to one of them and getting far from the other one.
So at least one of them cannot be sink. It follows that $G$ has at most one sink. A similar argument shows that $G$ has at most one source.

Now assume that $T$ is a directed tree and has a source $w_0$ and a sink $z_0$.
Let $P$ be the path between $w_0$ and $z_0$ in $T$.
Since $w_0$ is source and $z_0$ is sink, $P$ is a one-way path directed from $w_0$ to $z_0$. 
We claim that $T$ is the same as the path $P$.
Assume not. 
Hence, $T$ has an edge $e$ which is not an edge of $P$ but is adjacent to exactly one of vertices of $P$. 
Now one can verify that depending on the direction of $e$, either $e$ is getting closer to both vertices $w_0$ and $z_0$ or is getting far from both of them. But both cases contradict the assumption that $w_0$ is a source and $z_0$ is a sink.
So one sees that the only directed tree having both source and sink is a one-way path. \hfill $\square$
\end{proof}

\begin{lem}\label{indualpartcubewithGFtreesomeproperties}
Assume that $(X,\mathcal{F})$ is a dual well-graded family. Also assume that $G_{\mathcal{F}}$ is a directed tree with distinct labels.
Then the following hold.
\begin{enumerate}
\item{Every vertex of degree at least $3$ of $G_{\mathcal{F}}$ is either a source or a sink.}\label{indualpartcubewithGFtreeeveryvertexdeg3orbiggerissourceorsink}
\item{If a vertex $w$ of $G_{\mathcal{F}}$ of degree at least 2 is a source, then $\emptyset \in \mathcal{F}$ and 
$w=v_{\emptyset}$. Similarly, if a vertex $w$ of degree at least 2 is a sink, then $X \in \mathcal{F}$ and 
$w=v_X$.}\label{indualpartcubewithGFtreSourceSinkcorrespondemptyamdX}
\end{enumerate}
\end{lem}
\begin{proof}
\ref{indualpartcubewithGFtreeeveryvertexdeg3orbiggerissourceorsink}) Let $w_0$ be a vertex of degree at least 3 in $G_{\mathcal{F}}$. 
Without loss of generality, we can assume that there is an incoming edge to $w_0$ since otherwise 
we can work with $G_{\overline{\mathcal{F}}}$ and prove the statement for it instead of $G_{\mathcal{F}}$ while in this case by Remark \ref{reverseofoneincgraph} we are sure that $w_0$ has incoming edges in $G_{\overline{\mathcal{F}}}$ (in fact in this case every edge connected to $w_0$ would be incoming) and by Remark \ref{selfanddualpartcubeiffcomplementselfanddualpartcube},
$(X,\overline{\mathcal{F}})$ is a dual well-graded family. 
Also it is easy to see that if the statement holds for $G_{\overline{\mathcal{F}}}$, then it holds for $G_{\mathcal{F}}$ too.
So it would be enough to consider only the case that $w_0$ has incoming edge(s) in $G_{\mathcal{F}}$.

Assume that $e$ is an incoming edge to $w_0$. As usual, we let $lab(e)$ be the label of $e$ in $G_{\mathcal{F}}$ which is in fact a member of $X$.
Consider $\mathcal{A}_{lab(e)}$ (as defined in the definition of dual of systems) and identify it with its corresponding subset of vertices of $G_{\mathcal{F}}$ (as explained in Remark \ref{identifyingY_Fwithsetofverticesof1-inc}) and denote the subgraph induced on it by $G_e$. Also denote the subgraph induced on its complement in the set of vertices of $G_{\mathcal{F}}$ by $H_e$. 
Since the labels of edges of $G_{\mathcal{F}}$ are distinct, $e$ is the only edge with label $lab(e)$. Thus, it is easily seen that $G_e$ is the same as that connected component of $G_{\mathcal{F}} \setminus \{e\}$ to which the direction of $e$ is towards.
So, $G_e$ is a tree. Also $H_e$ is a tree.

\vspace{1mm}

\textit{Claim 1.} Directions of edges of $G_e$ are getting closer to $w_0$ in $G_{\mathcal{F}}$ (see Definition \ref{defgettingcloserfar}).

\textit{Proof of Claim 1.} Assume not. Then, since $G_{\mathcal{F}}$ is a tree and by Remark \ref{indirectedtreeeveryedgegetingcloseorfar}, there is an edge $e'$ in $G_e$ that is getting far from $w_0$. 
Now one can see that $\mathcal{A}_{lab(e')} \subsetneqq \mathcal{A}_{lab(e)}$ and since the degree of $w_0$ is at least 3, we have $|\mathcal{A}_{lab(e)} \setminus \mathcal{A}_{lab(e')}|\geqslant 2$.
So since the dual system is assumed to be a well-graded family, there exists some $\mathcal{A}_{lab(e'')}$ in the dual system for 
some edge $e''$ different from $e$ and $e'$ such that 
$\mathcal{A}_{lab(e')} \subset \mathcal{A}_{lab(e'')} \subset \mathcal{A}_{lab(e)}$ and $|\mathcal{A}_{lab(e)} \setminus \mathcal{A}_{lab(e'')}|=1$. 
Denote the subgraph of $G_{\mathcal{F}}$ induced on the vertex set $\mathcal{A}_{lab(e'')}$ by $G_{e''}$.
Since $\mathcal{A}_{lab(e'')} \subset \mathcal{A}_{lab(e)}$, $G_{e''}$ is a subgraph of $G_e$.
Also it is easily seen that $G_{e''}$ is one of the two connected components of $G_{\mathcal{F}} \setminus \{e''\}$, namely the one to which the direction of $e''$ is towards.
Since $G_{e''}$ is a subgraph of $G_e$, it is not hard to see that $e''$ must be an edge of $G_e$.
Since the degree of $w_0$ is at least 3, there are at least two vertices $w_1$ and $w_2$ in $G_e$ which are adjacent to $w_0$.
But it is not hard to see that $G_{e''}$
does not contain $w_0$ and at least one of $w_1$ or $w_2$ among its vertices. Hence, $|\mathcal{A}_{lab(e)} \setminus \mathcal{A}_{lab(e'')}| \geqslant 2$.
It follows that such a $\mathcal{A}_{lab(e'')}$ with mentioned properties does not exist
and this is a contradiction.
So all edges of $G_e$ are getting closer to $w_0$ in $G_{\mathcal{F}}$.
\hfill \textit{Claim 1} $\square$

\vspace{1mm}

\textit{Claim 2.} Directions of edges of $H_e$ are getting closer to $w_0$ in $G_{\mathcal{F}}$.

\textit{Proof of Claim 2.}
Let $q$ be one of the edges connected to $w_0$ and different from $e$.
One can see that $q$ is an edge in $G_e$ and by Claim 1, it is incoming to $w_0$.
Let $G_{q}$ be the subgraph induced from $G_{\mathcal{F}}$ on the vertex set $\mathcal{A}_{lab(q)}$.
Since $q$ is incoming to $w_0$, one can repeat the argument of Claim 1 for the graph $G_q$ instead of $G_e$ and the edge $q$ instead of $e$ and show that directions of edges of $G_{q}$ are getting closer to $w_0$ in $G_{\mathcal{F}}$. 
On the other hand, one can see that $H_e$ is a subgraph of $G_q$. 
Thus, every edge of $H_e$ is an edge of $G_q$.
So the directions of edges of $H_e$ are getting closer to $w_0$ in $G_{\mathcal{F}}$.
\hfill \textit{Claim 2} $\square$

\vspace{1.5mm}

Let $E$ be the union of sets of edges of $G_e$ and $H_e$. It is easy to see that $E$ contains every edge of $G_{\mathcal{F}}$ except $e$. But $e$ is incoming to $w_0$ and by Claim 1 and Claim 2, the directions of edges in $E$ are getting closer to $w_0$ in $G_{\mathcal{F}}$. 
So the direction of every edge of $G_{\mathcal{F}}$ is getting closer to $w_0$ in $G_{\mathcal{F}}$. Therefore, $w_0$ is a sink and the proof is complete.
Note that if at the beginning of the proof we had assumed that there is no incoming edge to $w_0$, then as mentioned earlier, we would have worked with $G_{\overline{\mathcal{F}}}$ and would have proved that $w_0$ is a sink in $G_{\overline{\mathcal{F}}}$. 
Therefore, in this case $w_0$ would be a source in $G_{\mathcal{F}}$.

\vspace{1mm}

\ref{indualpartcubewithGFtreSourceSinkcorrespondemptyamdX})
Let $w \in V(G_{\mathcal{F}})$ be a source of degree at least 2. 
We claim that $\emptyset \in \mathcal{F}$ and $st(w)=\emptyset$.
Assume for contradiction that there is some $x \in st(w)$. Then, since $w$ is a source, one can see that $x$ belongs to every member of $\mathcal{F}$. 
It follows that $\mathcal{A}_x=Y_{\mathcal{F}}$ (where these notations were defined in definition of dual of set systems). Thus,  $Y_{\mathcal{F}} \in \mathcal{F}^*$. 
Now since $(Y_{\mathcal{F}},\mathcal{F}^*)$ is a well-graded family and $\mathcal{F}^*$ in addition of $Y_{\mathcal{F}}$ contains also members with size strictly less than the size of $Y_{\mathcal{F}}$, one sees that there exists some member of $\mathcal{F}^*$, say $\mathcal{A}_z$ for some $z \in ess_{\mathcal{F}}(X)$, which being viewed as a subset of $Y_{\mathcal{F}}$, has size exactly one less than the size of $Y_{\mathcal{F}}$, namely $|\mathcal{F}|-1$.
By Remark \ref{|ess(X)|numberofdistinctlabels}, $z$ appears as the label of at least one edge of the graph $G_{\mathcal{F}}$. But since by assumption the labels of edges of $G_{\mathcal{F}}$ are distinct, $z$ appears as the label of exactly one edge, say $e$, of $G_{\mathcal{F}}$. 
Identify $\mathcal{A}_z$ with its corresponding subset of vertices of $G_{\mathcal{F}}$ (as explained in Remark \ref{identifyingY_Fwithsetofverticesof1-inc}). 
One sees that $\mathcal{A}_z$ corresponds to the vertex set of that connected component of $G_{\mathcal{F}} \setminus \{e\}$ to which the direction of $e$ is towards.
On the other hand, since $w$ has degree at least 2 and $G_{\mathcal{F}}$ is a tree, $G_{\mathcal{F}} \setminus \{w\}$ has at least two connected components and the number of vertices in each of them is less than or equal to $|\mathcal{F}|-2$.
Moreover, since $w$ is a source, the edge $e$ is getting far from $w$ and it is not hard to see that $\mathcal{A}_z$ is a subset of vertices of one of the connected components of $G_{\mathcal{F}} \setminus \{w\}$.
Thus we have $|\mathcal{A}_z| \leqslant |\mathcal{F}|-2$. 
But this is a contradiction.
So $st(w)=\emptyset$.
Therefore, $\emptyset \in \mathcal{F}$ and 
$w=v_{\emptyset}$. 
In the case that $w$ is a sink, one can similarly show that $X \in \mathcal{F}$ and $w=v_X$.
\hfill $\square$
\end{proof}

\begin{dfn}
By a \textit{full-chain} of length $n$ we mean a sequence of sets $A_1 \subseteq A_2 \ldots \subseteq A_n$ such that for each $2\leqslant i \leqslant n$ we have $|A_i \setminus A_{i-1}|=1$. We call a set system $(X,\mathcal{F})$ a full-chain system if 
$\mathcal{F}$ can be ordered in a way that becomes a full-chain.
\end{dfn}
Note that in the definition of full-chain systems, we allow the domain of the set system to possibly contain elements which belong to every or to no member of $\mathcal{F}$.

\begin{dfn}\label{defstarliketree}
We call an edge-labelled directed tree a \textit{starlike} if all labels of it are distinct and it has a distinguished vertex $v_0$ of degree at least $2$ which is either a source or a sink and moreover, other vertices have degrees at most $2$. 
In this case, we call $v_0$ the centre of the starlike. 
We call each of the maximal directed paths starting from $v_0$ (when $v_0$ is a source), or ending in $v_0$ (when $v_0$ is a sink) a wing of the starlike.
\end{dfn}

\begin{dfn}\label{defupwarddownwardstarlike}
We call a set system $(X,\mathcal{F})$ upward-starlike if $G_{\mathcal{F}}$ is a starlike tree, $\emptyset \in \mathcal{F}$ and $v_{\emptyset}$ is the centre and source of $G_{\mathcal{F}}$ 
and in addition, 
there is some element of 
$X$ which does not belong to any member of $\mathcal{F}$ (or equivalently $\emptyset$ belongs to the dual system $\mathcal{F}^*$).
Similarly, we call $(X,\mathcal{F})$ downward-starlike if $G_{\mathcal{F}}$ is a starlike tree, $X \in \mathcal{F}$ 
and $v_X$ is the centre and sink of $G_{\mathcal{F}}$
and moreover, there is an element of $X$ 
which belongs to every member of $\mathcal{F}$ (or equivalently $Y_{\mathcal{F}}$ belongs to the dual system $\mathcal{F}^*$).
As an analogue to the notion of wing in starlike-trees, we call each of the maximal full-chains starting from element $\emptyset$ in an upward-starlike system, or ending in $X$ in a downward-starlike system a wing of that system.
\end{dfn}
It is easy to see that a set system $(X,\mathcal{F})$ is upward-starlike if and only if the set system $(X,\overline{\mathcal{F}})$ (see Definition \ref{defFbarreverse}) is downward-starlike.

\begin{rem}\label{factsaboutfullchainupwarddownwardstarlike}
\begin{enumerate}
\item{Every full-chain system, upward-starlike system or downward-starlike system satisfies all equivalent clauses of Theorem \ref{characterizeextvc=1withequalities,right=impliesleft=}.}\label{everyupordownstarlikesatisfiesTheorem} 

\item{The dual system of every full-chain system, upward-starlike system and downward-starlike system is a full-chain system, upward-starlike system and downward-starlike system respectively.}\label{fullchainupwarddownwardstarlikearealmostselfdual}

\item{Every full-chain system, upward-starlike system and downward-starlike system is self-and dual extremal (and so self-and dual well-graded).}\label{everyupordownstarlikeisselfanddualextremal}

\item{No upward-starlike or downward-starlike system is a maximum system.}\label{noupwarddownwardstarlikeismaximum}

\item{Let $(X,\mathcal{F})$ be a full-chain system of length at least $2$. Then $(X,\mathcal{F})$ is maximum if and only if $X=ess_{\mathcal{F}}(X)$. 
Moreover, either of these equivalent statements for a full-chain system of length at least $2$, implies that the first member of the chain is $\emptyset$.}\label{fullchainlengthatleast2maximumiff}
\end{enumerate}
\end{rem}
\begin{proof}
\ref{everyupordownstarlikesatisfiesTheorem}) By definition, the one-inclusion graph of every upward-starlike or downward-starlike system is a tree with distinct labels. 
Also clearly the one-inclusion graph of every full-chain system is a one-way path with distinct labels. 
So by Theorem \ref{characterizeextvc=1withequalities,right=impliesleft=} (\ref{GFistreealllabelsaredifferent}), the result is clear. 

\ref{fullchainupwarddownwardstarlikearealmostselfdual}) It is not very hard to verify the statement by direct computing the dual of such systems.

\ref{everyupordownstarlikeisselfanddualextremal}) By combination of parts \ref{everyupordownstarlikesatisfiesTheorem} and \ref{fullchainupwarddownwardstarlikearealmostselfdual} of this remark, the result is clear.

\ref{noupwarddownwardstarlikeismaximum})
One can see that every upward-starlike system $(X,\mathcal{F})$ has positive VC-dimension. So if it is maximum, then every element of $X$ as a one-element subset is shattered by $\mathcal{F}$. But this is impossible since there is some element in $X$ which belongs to no member of $\mathcal{F}$. So every upward-starlike system is not maximum.
A similar argument works for downward-starlike systems.

\ref{fullchainlengthatleast2maximumiff})
First assume that the full-chain system $(X,\mathcal{F})$ is maximum and $|\mathcal{F}| \geqslant 2$. So $VCdim(X,\mathcal{F})\geqslant 1$. Therefore, since the system is maximum, every element of $X$ as a one-element subset is shattered by $\mathcal{F}$. Thus, $X=ess_{\mathcal{F}}(X)$.
Let $A$ be the first member of the full-chain $\mathcal{F}$. 
So each element of $A$ belongs to all members of $\mathcal{F}$.
Hence, $A \subseteq X \setminus ess_{\mathcal{F}}(X)$. 
It follows that $A=\emptyset$ and $\mathcal{F}$ starts from $\emptyset$.
Furthermore, it is easy to see that every full-chain system $(X,\mathcal{F})$ of length at least $2$ with $X=ess_{\mathcal{F}}(X)$ starts from $\emptyset$ and is a $1$-maximum system.
\hfill $\square$
\end{proof}

\vspace{1mm}

Note that the converse of Remark \ref{factsaboutfullchainupwarddownwardstarlike}(\ref{everyupordownstarlikesatisfiesTheorem}) does not hold. A simple example showing this is the set system $(X,\mathcal{F})$ where 
$X:=\{1,2,3\}$ and $\mathcal{F}:=\{\{1\},\{1,2\},\{1,3\}\}$. 
In this system $G_{\mathcal{F}}$ is a tree with distinct labels on its edges.
So this system satisfies Clause \ref{GFistreealllabelsaredifferent} of the Theorem \ref{characterizeextvc=1withequalities,right=impliesleft=}. But it is clearly neither a full-chain system, nor an upward-starlike system nor a downward-starlike system.
Regarding Remark \ref{factsaboutfullchainupwarddownwardstarlike}(\ref{fullchainupwarddownwardstarlikearealmostselfdual}), one can see that 
every upward-starlike system or downward-starlike system is an almost self-dual system. 
Also if a full-chain system starts with a nonempty set and does not have any element in its domain which appears in no member of the chain, then its dual system is a full-chain with the same length.

\begin{lem}\label{selfanddualpathdensefortreeandsemitree}
Let $(X,\mathcal{F})$ be a set system. 
Then the following hold. 
\begin{enumerate}
\item{The system $(X,\mathcal{F})$ is dual well-graded and $G_{\mathcal{F}}$ is a directed tree with distinct labels if and only if $(X,\mathcal{F})$ is either a full-chain system (or equivalently $G_{\mathcal{F}}$ is a one-way path) 
or a downward-starlike or an upward-starlike system.
}\label{1-inctreesystemsdualpathdenseiffarepathorstarlike}
\item{The system $(X,\mathcal{F})$ is ess-dual well-graded and $G_{\mathcal{F}}$ is a directed tree with distinct labels
if and only if $(X,\mathcal{F})$ is a full-chain system. 
}
\label{1-inctreesystemsessdualpathdenseiffarepath}
\item{Assume that $G_{\mathcal{F}}$ is a semitree. 
Then the dual system and ess-dual system of $(X,\mathcal{F})$ are not well-graded.}
\label{dualofsemitreeisnotpathdense}
\end{enumerate}
\end{lem}
\begin{proof}
\ref{1-inctreesystemsdualpathdenseiffarepathorstarlike})
If a set system $(X,\mathcal{F})$ is either a full-chain system 
or a downward-starlike or an upward-starlike, then by Remark
\ref{factsaboutfullchainupwarddownwardstarlike}(\ref{everyupordownstarlikeisselfanddualextremal}), the system is self-and-dual well-graded family.

For the converse, assume that $(X,\mathcal{F})$ is a dual well-graded family and $G_{\mathcal{F}}$ a directed tree with distinct labels.
We divide the situations into two cases.

\vspace{1mm}

\textit{Case I.} In this case we assume that the maximum degree of vertices of $G_{\mathcal{F}}$ is less than or equal to $2$. 
Then, since $G_{\mathcal{F}}$ is a tree, its underlying undirected graph must be a path.
If there is no vertex of degree 2 or if in every vertex with degree 2 the incoming and outgoing degrees are 1, then $G_{\mathcal{F}}$ is a one-way path which follows that the system is a full-chain system. 
So assume that there is some vertex with incoming or outgoing degree 2.
First assume that there is some vertex $w_0$ with incoming degree 2. We claim that $w_0$ is a sink. 
Assume for contradiction that $w_0$ is not a sink. 
Let $e_1$ and $e_2$ be the incoming edges to $w_0$.
Now consider $\mathcal{A}_{lab(e_1)}$ and $\mathcal{A}_{lab(e_2)}$ as two subsets of vertices of $G_{\mathcal{F}}$ (as explained in Remark \ref{identifyingY_Fwithsetofverticesof1-inc}) and denote by $G_{e_1}$ and $G_{e_2}$ the subgraphs induced on them by $G_{\mathcal{F}}$. 
Since by assumption labels of edges of $G_{\mathcal{F}}$ are distinct, $e_1$ and $e_2$ are the only edges with labels $lab(e_1)$ and $lab(e_2)$ respectively.
It is easy to see that $G_{e_1}$ and $G_{e_2}$ are two paths, union of them is $G_{\mathcal{F}}$ and they have only vertex $w_0$ in common. 
Since $w_0$ is not sink and by using Remark \ref{indirectedtreeeveryedgegetingcloseorfar}, there is an edge $e_3$ 
whose direction is getting far from $w_0$. 
Without loss of generality, we may assume that $e_3$ is in $G_{e_1}$.
Now it is not difficult to see that 
$\mathcal{A}_{lab(e_3)} \subset \mathcal{A}_{lab(e_1)}$ and 
$|\mathcal{A}_{lab(e_1)} \setminus \mathcal{A}_{lab(e_3)}|\geqslant 2$. 
Since $(Y_{\mathcal{F}},\mathcal{F}^*)$ is assumed to be a well-graded family and 
$\mathcal{A}_{lab(e_1)}, \mathcal{A}_{lab(e_3)} \in \mathcal{F}^*$,
there exists $B \in \mathcal{F}^*$ such that $\mathcal{A}_{lab(e_3)} \subsetneqq B \subsetneqq \mathcal{A}_{lab(e_1)}$ and 
$|\mathcal{A}_{lab(e_1)} \setminus B|=1$.
So $B=\mathcal{A}_x$ for some $x \in X$ and since $B \not = \emptyset$ and $B \not = Y_{\mathcal{F}}$, we have $x \in ess_{\mathcal{F}}(X)$.
Hence, by using Remark \ref{|ess(X)|numberofdistinctlabels}, we have $B=\mathcal{A}_{lab(e_4)}$ for some edge $e_4$ of $G_{\mathcal{F}}$.
Therefore, $\mathcal{A}_{lab(e_3)} \subsetneqq \mathcal{A}_{lab(e_4)} \subsetneqq \mathcal{A}_{lab(e_1)}$ and 
$|\mathcal{A}_{lab(e_1)} \setminus \mathcal{A}_{lab(e_4)}|=1$.
Note that $e_4$ is the only edge with label $lab(e_4)$. One can see that $\mathcal{A}_{lab(e_4)}$ corresponds to the vertex set of that connected component of $G_{\mathcal{F}} \setminus \{e_4\}$ to which the direction of $e_4$ is towards.
But now it is not hard to see that $e_4 \in G_{e_1}$ and $e_4$ is getting far from $w_0$. Also it is easily seen that $|\mathcal{A}_{lab(e_1)} \setminus \mathcal{A}_{lab(e_4)}|\geqslant 2$ which is a contradiction.
Therefore, $w_0$ is a sink.
So $G_{\mathcal{F}}$ is a starlike with centre $w_0$.
Also by Lemma \ref{indualpartcubewithGFtreesomeproperties}(\ref{indualpartcubewithGFtreSourceSinkcorrespondemptyamdX}), $X \in \mathcal{F}$ and $w_0=v_X$. 
Now we only need to show that there is an element of $X$ which belongs to every member of $\mathcal{F}$.
Let $a_1$ and $a_2$ be two ending vertices of the underlying undirected graph of $G_{\mathcal{F}}$ (which is a path) and $h_1$ and $h_2$ be the unique adjacent edges to them respectively. Since $G_{\mathcal{F}}$ is starlike with centre $w_0$ which is a sink, it is easily seen that $\mathcal{A}_{lab(h_1)}=Y_{\mathcal{F}} \setminus \{a_1\}$ and $\mathcal{A}_{lab(h_2)}=Y_{\mathcal{F}} \setminus \{a_2\}$
where we recall that we are again identifying the elements of $Y_{\mathcal{F}}$ with the set of vertices of the graph $G_{\mathcal{F}}$ as explained in Remark \ref{identifyingY_Fwithsetofverticesof1-inc}.
Thus, 
$|\mathcal{A}_{lab(h_1)} \triangle \mathcal{A}_{lab(h_2)}|=2$. Since $(Y_{\mathcal{F}},\mathcal{F}^*)$ is assumed to be a well-graded family and $\mathcal{A}_{lab(h_1)}, \mathcal{A}_{lab(h_2)} \in \mathcal{F}^*$, there must exist
some $c \in X$ such that $|\mathcal{A}_{lab(h_1)} \triangle \mathcal{A}_c|=1$ and 
$|\mathcal{A}_c \triangle \mathcal{A}_{lab(h_2)}|=1$.
So either $\mathcal{A}_c=Y_{\mathcal{F}}$ or $\mathcal{A}_c=Y_{\mathcal{F}} \setminus \{a_1,a_2\}$. 
It follows that $c \in X \setminus ess_{\mathcal{F}}(X)$ since otherwise, by Remark \ref{|ess(X)|numberofdistinctlabels} and the assumption that labels of edges of $G_{\mathcal{F}}$ are distinct, $c$ would be the label of exactly one edge of $G_{\mathcal{F}}$ and then in this case, one can see that $\mathcal{A}_c$ would contain exactly one of $a_1$ or $a_2$ which would be a contradiction. Consequently, since $\mathcal{A}_c \not = \emptyset$, we have 
$\mathcal{A}_c=Y_{\mathcal{F}}$.
Hence, $c$ is an element of $X$ which belongs to every member of $\mathcal{F}$. Therefore, $(X,\mathcal{F})$ is a downward-starlike system.

With a similar argument as above, one can show that if there is some vertex $w_0$ with outgoing degree 2, then $w_0$ is a source and $G_{\mathcal{F}}$ is a starlike with centre $w_0$, $\emptyset \in \mathcal{F}$ and 
$w_0=v_{\emptyset}$. 
Moreover, there is some element in $X$ which does not belong to any member of $\mathcal{F}$. Therefore, in this case $(X,\mathcal{F})$ is an upward-starlike system.

\vspace{1mm}

\textit{Case II.} 
In this case we assume that there is some vertex $w_0$ of $G_{\mathcal{F}}$ with degree at least $3$. 
So by using Lemma
\ref{indualpartcubewithGFtreesomeproperties}(\ref{indualpartcubewithGFtreeeveryvertexdeg3orbiggerissourceorsink}), $w_0$ is either a source or a sink.
By Remark \ref{sourcesinkindirectedtrees}, except $w_0$ there is no other source or sink in $G_{\mathcal{F}}$.
Hence, again by Lemma
\ref{indualpartcubewithGFtreesomeproperties}(\ref{indualpartcubewithGFtreeeveryvertexdeg3orbiggerissourceorsink}) there is no vertex of degree at least $3$ in $G_{\mathcal{F}}$ except $w_0$.
Now one can see that the graph $G_{\mathcal{F}}$ is starlike.
Also by Lemma \ref{indualpartcubewithGFtreesomeproperties}(\ref{indualpartcubewithGFtreSourceSinkcorrespondemptyamdX}), if $w_0$ is a source, then $\emptyset \in \mathcal{F}$ and $w_0=v_{\emptyset}$. Similarly, if $w_0$ is a sink, then $X \in \mathcal{F}$ and $w_0=v_X$.

\vspace{0.75mm}

\textit{Claim:} If $w_0$ is a source, then 
$\mathcal{F}^*$ contains $\emptyset$.
Also if $w_0$ is a sink, 
$\mathcal{F}^*$ contains $Y_{\mathcal{F}}$.

\textit{Proof of Claim:}
We only prove the former case. The proof of the other case is similar.
Assume that $w_0$ is a source.
Let $a_1$ and $a_2$ be two vertices of degree 1 of $G_{\mathcal{F}}$
and $h_1$ and $h_2$ be the unique adjacent edges to them respectively.
We remind that $G_{\mathcal{F}}$ is starlike and one sees that $a_1$ and $a_2$ belong to the different wings of $G_{\mathcal{F}}$. Since $G_{\mathcal{F}}$ is starlike with centre $w_0$ which is a source, it is easily seen that $\mathcal{A}_{lab(h_1)}=\{a_1\}$ and $\mathcal{A}_{lab(h_2)}=\{a_2\}$
where we recall that we are identifying the elements of $Y_{\mathcal{F}}$ with the set of vertices of the graph $G_{\mathcal{F}}$ as explained in Remark \ref{identifyingY_Fwithsetofverticesof1-inc}. 
Since $(Y_{\mathcal{F}},\mathcal{F}^*)$ is assumed to be a well-graded family and $\mathcal{A}_{lab(h_1)}, \mathcal{A}_{lab(h_2)} \in \mathcal{F}^*$, there must exist
some $c \in X$ such that $|\mathcal{A}_{lab(h_1)} \triangle \mathcal{A}_c|=1$ and $|\mathcal{A}_c \triangle \mathcal{A}_{lab(h_2)}|=1$.
So either $\mathcal{A}_c=\emptyset$ or $\mathcal{A}_c=\{a_1,a_2\}$. 
We claim that $c \in X \setminus ess_{\mathcal{F}}(X)$. Assume not. Then by Remark \ref{|ess(X)|numberofdistinctlabels} and the assumption that labels of edges of $G_{\mathcal{F}}$ are distinct, $c$ is the label of exactly one edge of $G_{\mathcal{F}}$. So $\mathcal{A}_c$ is nonempty which follows that $\mathcal{A}_c=\{a_1,a_2\}$. Also since $a_1$ and $a_2$ belong to the different wings of $G_{\mathcal{F}}$ and $w_0$ is source, 
one sees that $\mathcal{A}_c$ contains at most one of $a_1$ or $a_2$ which is a contradiction. Consequently, $c \in X \setminus ess_{\mathcal{F}}(X)$.
Now, since $\mathcal{A}_c \not = Y_{\mathcal{F}}$, we have 
$\mathcal{A}_c=\emptyset$. Therefore, $\mathcal{F}^*$ contains $\emptyset$.
\hfill \textit{Claim 2} $\square$

\vspace{0.75mm}

Using the above claim, if $w_0$ is a source, then $\mathcal{F}^*$ contains $\emptyset$ which is equivalent to say that there exists some $x \in X$ which belongs to no members of $\mathcal{F}$. Therefore, in this case the system $(X,\mathcal{F})$ is upward-starlike. Similarly, if $w_0$ is a sink, 
then $Y_{\mathcal{F}} \in \mathcal{F}^*$ which is equivalent to say that
there exists some $x \in X$ which belongs to every member of $\mathcal{F}$. In this case the system is downward-starlike. 

\vspace{1.5mm}

\ref{1-inctreesystemsessdualpathdenseiffarepath})
If the set system $(X,\mathcal{F})$ is a full-chain system, then it is easy to see that it is ess-dual well-graded and also $G_{\mathcal{F}}$ is a directed tree (in fact a one-way path) with distinct labels. We prove the converse.
For simplifying the notations, we use notation $Z$ for $ess_{\mathcal{F}}(X)$. Consider the set system $(Z,\mathcal{G})$ where 
$\mathcal{G}=\mathcal{F} \cap Z$, the trace of the system $(X,\mathcal{F})$ on the subset $Z$ of its domain $X$. 
One can see that the graphs $G_{\mathcal{F}}$ and $G_{\mathcal{G}}$ are isomorphic (as labelled directed graphs) via the map $v_A \rightarrow v_{A \cap Z}$ for every $A \in \mathcal{F}$. So $G_{\mathcal{G}}$ is a directed tree with distinct labels.
One can also verify that $Y_{\mathcal{F}}$ and $Y_{\mathcal{G}}$ 
can be identified and moreover, ess-dual of $(X,\mathcal{F})$ and dual of $(Z,\mathcal{G})$ can be seen as the same.
Hence, $(Z,\mathcal{G})$ is a dual well-graded family.
So by part \ref{1-inctreesystemsdualpathdenseiffarepathorstarlike} of this lemma, $(Z,\mathcal{G})$ is either a full-chain system or a downward-starlike system or an upward-starlike system. 
But $(Z,\mathcal{G})$ cannot be upward-starlike or downward-starlike since $Z=ess_{\mathcal{G}}(Z)$ while in the definition of upward (or downward)-starlike systems, there is some element which belongs to no (every) member of the system which implies that this element belongs to $Z \setminus ess_{\mathcal{G}}(Z)$. So $(Z,\mathcal{G})$ is a full-chain system. Now one can see that the initial system $(X,\mathcal{F})$ is also a full-chain system.

\vspace{2mm}

\ref{dualofsemitreeisnotpathdense})
Assume for contradiction that the dual system $(Y_{\mathcal{F}},\mathcal{F}^*)$ is a well-graded family.
Let $a_1, a_2, a_3$ and $a_4$ be the vertices of the parallel-directed-labelled $C_4$-cycle of $G_{\mathcal{F}}$ indexed by the cyclic order of the $C_4$ cycle. Denote this $C_4$-cycle by $C$ .
It is easy to see that without loss of generality, one can assume that 
the directions of edges of the cycle $C$ are $a_1$ to $a_2$, $a_2$ to $a_3$, $a_4$ to $a_3$ and $a_1$ to $a_4$. 
Note that by definition of semitree graphs we have $lab(a_2a_3)=lab(a_1a_4)$ and $lab(a_1a_2)=lab(a_4a_3)$. 
Let $l:=lab(a_2a_3)$ and $s:=lab(a_1a_2)$. So $l$ and $s$ are in fact two distinct elements of $ess_{\mathcal{F}}(X)$. 
Also $\mathcal{A}_{l} \in \mathcal{F}^*$. 
Moreover, $a_3, a_4 \in \mathcal{A}_{l}$ and $a_1, a_2 \not \in \mathcal{A}_{l}$
where we are identifying the elements of $Y_{\mathcal{F}}$ with the set of vertices of the graph $G_{\mathcal{F}}$ as explained in Remark \ref{identifyingY_Fwithsetofverticesof1-inc}.
Similarly, we have $a_2, a_3 \in \mathcal{A}_s$ and $a_1, a_4 \not \in \mathcal{A}_s$.
Now since $\mathcal{F}^*$ contains other elements except $\mathcal{A}_{l}$ and $(Y_{\mathcal{F}},\mathcal{F}^*)$ is assumed to be a well-graded family, one can see that there exists some $\mathcal{A}_{t} \in \mathcal{F}^*$ for some $t \in X$ different from $l$ such that $|\mathcal{A}_{l} \triangle \mathcal{A}_{t}|=1$.
If $t \in X \setminus ess_{\mathcal{F}}(X)$, then either $\mathcal{A}_{t}=Y_{\mathcal{F}}$ or $\mathcal{A}_{t}=\emptyset$. In the former case we get the contradiction $|\mathcal{A}_{l} \triangle \mathcal{A}_{t}| \geqslant 2$ since as mentioned earlier,
$a_1, a_2 \in Y_{\mathcal{F}} \setminus \mathcal{A}_l$ and in the latter case 
we get the same contradiction since
$a_3, a_4 \in \mathcal{A}_l$. 
Thus, $t \in ess_{\mathcal{F}}(X)$.
Note that by Theorem $\ref{structureofpathdensewithF=ess+2} (\ref{1-incgraphissemitree} \Rightarrow \ref{pathdensewithF=ess+2})$, the system $(X,\mathcal{F})$ is a well-graded family.
So by Remark \ref{|ess(X)|numberofdistinctlabels}, there is an edge $e$ of the graph $G_{\mathcal{F}}$ with label $t$. 
If $t = s$, then $\mathcal{A}_t=\mathcal{A}_s$ and so $a_2, a_3 \in \mathcal{A}_t$ and $a_1, a_4 \not \in \mathcal{A}_t$. 
Combining this with the fact that $a_3, a_4 \in \mathcal{A}_l$ and $a_1, a_2 \not \in \mathcal{A}_l$, we have $a_2, a_4 \in \mathcal{A}_{l} \triangle \mathcal{A}_{t}$.
It follows that 
$|\mathcal{A}_{l} \triangle \mathcal{A}_{t}| \geqslant 2$ which is contradiction. Therefore, $t \not =s$ and as mentioned earlier, $t \not =l$. Hence, the edge $e$ is not among four edges of the cycle $C$ of $G_{\mathcal{F}}$. 
Thus, by the definition of a semitree, $e$ is the only edge of $G_{\mathcal{F}}$ which has $t$ as the label.
So by Remark \ref{inpartcubedgeswithlabelaiscutsets}, $\{e\}$ is the same as the cut-set $E(\mathcal{A}_t,\mathcal{A}_t^c)$ of the graph $G_{\mathcal{F}}$ 
and the graph $G_{\mathcal{F}} \setminus \{e\}$ has two connected components which are in fact the induced subgraphs of $G_{\mathcal{F}}$ on $\mathcal{A}_t$ and $\mathcal{A}_t^c$. 
Also it is not hard to see that depending on the direction of the edge $e$,
the cycle $C$ is a subgraph of one of these two connected components. Hence, the set $\{a_1, a_2, a_3, a_4\}$ is either a subset of $\mathcal{A}_t$ or a subset of $\mathcal{A}_t^c$.
But in either of the cases we have $|\mathcal{A}_{l} \triangle \mathcal{A}_{t}| \geqslant 2$ which is a contradiction. 
Therefore, the dual system $(Y_{\mathcal{F}},\mathcal{F}^*)$ is not a well-graded family.
With a similar argument as above, one can also show that the ess-dual system of $(X,\mathcal{F})$ is not a well-graded family.
\hfill $\square$
\end{proof}

\begin{lem}\label{dualselfextpathdenseimpliesalmostequality}
Let $(X,\mathcal{F})$ be a set system. Then the following hold.
\begin{enumerate}
\item{If $(X,\mathcal{F})$ is a self-and-dual 
well-graded family, then we have 
$|ess_{\mathcal{F}}(X)|+1 \leqslant |\mathcal{F}| \leqslant |ess_{\mathcal{F}}(X)|+r_{\mathcal{F}}+1$
where $r_{\mathcal{F}}$ is the number of elements of the set $\{\emptyset,X\}$ which belong to $\mathcal{F}$.
So in particular if $\mathcal{F}$ contains none of $\emptyset$ and $X$, then we have $|\mathcal{F}|=|ess_{\mathcal{F}}(X)|+1$ which means that the additionality of the system is $1$.}\label{selfanddualPartialcubeinequ}
\item{If $(X,\mathcal{F})$ is a self-and-ess-dual 
well-graded family, then we have 
$|\mathcal{F}|=|ess_{\mathcal{F}}(X)|+1$.
Moreover, the quantity $r'_{\mathcal{F}}:=|\{A \in \mathcal{F}: ess_{\mathcal{F}}(X) \subseteq A \ or \ A \cap ess_{\mathcal{F}}(X)=\emptyset\}|$ is equal to $2$.}\label{selfandessdualPartialcubeinequ} 
\end{enumerate}
\end{lem}
\begin{proof}
\ref{selfanddualPartialcubeinequ})
As usual, let $(Y_{\mathcal{F}},\mathcal{F}^*)$ be the dual system of $(X,\mathcal{F})$. 
For simplifying the notations of the proof, we use $Y$ instead of $Y_{\mathcal{F}}$.
By using Proposition \ref{FessEGFinequalities} and our assumption that the dual system is well-graded, we have 
$|ess_{\mathcal{F}^*}(Y)|+1 \leqslant |\mathcal{F}^*|$.
We remind that each element of $\mathcal{F}^*$ is of the form $\mathcal{A}_x$ for some $x \in X$. Moreover, for every $x \in X \setminus ess_{\mathcal{F}}(X)$ either $\mathcal{A}_x=Y$ or $\mathcal{A}_x=\emptyset$. So we have
$|\mathcal{F}^*| \leqslant |ess_{\mathcal{F}}(X)|+2$.
On the other hand, for every $A \in \mathcal{F}$, $y_{A}$ belongs to $Y \setminus ess_{\mathcal{F}^*}(Y)$ if and only if either $y_A \in \mathcal{A}_x$ for every $x \in X$ or $y_A \not \in \mathcal{A}_x$ for every $x \in X$. This is equivalent to say that either $A=X$ or $A=\emptyset$.
Thus, $|Y \setminus ess_{\mathcal{F}^*}(Y)|=r_{\mathcal{F}}$ where we recall that 
$r_{\mathcal{F}}=|\mathcal{F} \cap \{\emptyset, X\}|$. 
So since $ess_{\mathcal{F}^*}(Y) \subseteq Y$, we have $|Y|- |ess_{\mathcal{F}^*}(Y)|=r_{\mathcal{F}}$.
Hence, we have $|\mathcal{F}|-r_{\mathcal{F}} = |Y|-r_{\mathcal{F}}=|ess_{\mathcal{F}^*}(Y)|$. 
Now combining the above facts, we have
$|\mathcal{F}|-r_{\mathcal{F}} +1 = |ess_{\mathcal{F}^*}(Y)|+1 \leqslant |\mathcal{F}^*| \leqslant |ess_{\mathcal{F}}(X)|+2$
which implies that $|\mathcal{F}| \leqslant |ess_{\mathcal{F}}(X)|+r_{\mathcal{F}}+1$.
Since $(X,\mathcal{F})$ is a well-graded family, again by Proposition \ref{FessEGFinequalities} we have $|ess_{\mathcal{F}}(X)|+1 \leqslant |\mathcal{F}|$.
So $|ess_{\mathcal{F}}(X)|+1 \leqslant |\mathcal{F}| \leqslant |ess_{\mathcal{F}}(X)|+r_{\mathcal{F}}+1$ as desired.
Moreover, if $\mathcal{F}$ contains none of $\emptyset$ or $X$, then $r_{\mathcal{F}}=0$ and $|\mathcal{F}|=|ess_{\mathcal{F}}(X)|+1$ which means that the additionality of $(X,\mathcal{F})$ is $1$.

\vspace{1mm}

\ref{selfandessdualPartialcubeinequ})
The proof is similar to the proof of part \ref{selfanddualPartialcubeinequ}. For sake of clarity we give it in detail.
We recall that $(Y_{\mathcal{F}},\mathcal{F}^*_{ess})$ is the notation for ess-dual system of $(X,\mathcal{F})$. 
Similar to part \ref{selfanddualPartialcubeinequ}, we use $Y$ instead of $Y_{\mathcal{F}}$.
Since the ess-dual system is assumed to be well-graded, by using Proposition \ref{FessEGFinequalities} we have 
$|ess_{\mathcal{F}^*_{ess}}(Y)|+1 \leqslant |\mathcal{F}^*_{ess}|$.
We recall that each element of $\mathcal{F}^*_{ess}$ is of the form $\mathcal{A}_x$ for some $x \in ess_{\mathcal{F}}(X)$. 
So
$|\mathcal{F}^*_{ess}| \leqslant |ess_{\mathcal{F}}(X)|$.
On the other hand, for every $A \in \mathcal{F}$, $y_{A}$ belongs to $Y \setminus ess_{\mathcal{F}^*_{ess}}(Y)$ if and only if either $y_A \in \mathcal{A}_x$ for every $x \in ess_{\mathcal{F}}(X)$ or $y_A \not \in \mathcal{A}_x$ for every $x \in ess_{\mathcal{F}}(X)$. 
This is equivalent to say that either $ess_{\mathcal{F}}(X)\subseteq A$ or $A \cap ess_{\mathcal{F}}(X)=\emptyset$.
So $$|Y \setminus ess_{\mathcal{F}^*_{ess}}(Y)|=|\{A \in \mathcal{F}: ess_{\mathcal{F}}(X) \subseteq A \ or \ A \cap ess_{\mathcal{F}}(X)=\emptyset\}|=r'_{\mathcal{F}}.$$
Hence, since $ess_{\mathcal{F}^*_{ess}}(Y) \subseteq Y$, we have $|Y|- |ess_{\mathcal{F}^*_{ess}}(Y)|=r'_{\mathcal{F}}$.
Therefore, $|\mathcal{F}|-r'_{\mathcal{F}} = |Y|-r'_{\mathcal{F}}=|ess_{\mathcal{F}^*_{ess}}(Y)|$. 
Now combining the above facts we have
$$|\mathcal{F}|-r'_{\mathcal{F}} +1 = |ess_{\mathcal{F}^*_{ess}}(Y)|+1 \leqslant |\mathcal{F}^*_{ess}| \leqslant |ess_{\mathcal{F}}(X)|$$
which implies that $|\mathcal{F}| \leqslant |ess_{\mathcal{F}}(X)|+r'_{\mathcal{F}}-1$.
Since $\mathcal{F}$ is a well-graded family, again by Proposition \ref{FessEGFinequalities} we have $|ess_{\mathcal{F}}(X)|+1 \leqslant |\mathcal{F}|$.
So 
$$|ess_{\mathcal{F}}(X)|+1 \leqslant |\mathcal{F}| \leqslant |ess_{\mathcal{F}}(X)|+r'_{\mathcal{F}}-1.$$
It can be easily seen 
that 
$r'_{\mathcal{F}} \leqslant 2$.
Therefore, one concludes that 
$|\mathcal{F}|=|ess_{\mathcal{F}}(X)|+1$ and $r'_{\mathcal{F}}=2$.
\hfill $\square$

\end{proof}

\vspace{1.5mm}

The following theorem is the main result of this section and characterizes self-and-(ess-) dual well-graded families.

\begin{thm}\label{characterizationselfanddualpathdense}
Assume that $(X,\mathcal{F})$ is a set system. 
Then the following hold.
\begin{enumerate}
\item{The system $(X,\mathcal{F})$ is a self-and-dual well-graded family if and only if it is either a full-chain system or an upward-starlike system or a downward-starlike system. Moreover, every self-and-dual well-graded family has additionality 1.
}\label{selfanddualpathdenseiffchainorstrongstarlike}
\item{The system $(X,\mathcal{F})$ is self-and-ess-dual well-graded  
if and only if $\mathcal{F}$ is a full-chain system. 
Moreover, every self-and-ess-dual well-graded family has additionality 1.
}\label{selfedddualpathdenseiffchain}
\end{enumerate}
\end{thm}
\begin{proof}
\ref{selfanddualpathdenseiffchainorstrongstarlike})
If a set system is either a full-chain system, 
or a downward-starlike or an upward-starlike, then by Remark
\ref{factsaboutfullchainupwarddownwardstarlike}(\ref{everyupordownstarlikeisselfanddualextremal}), the system is self-and-dual well-graded family.
Now we prove the converse. Assume that the system is a self-and-dual well-graded family.
By Lemma \ref{dualselfextpathdenseimpliesalmostequality}(\ref{selfanddualPartialcubeinequ}), we have 
$|ess_{\mathcal{F}}(X)|+1 \leqslant |\mathcal{F}| \leqslant |ess_{\mathcal{F}}(X)|+r_{\mathcal{F}}+1$ where we recall that 
$r_{\mathcal{F}}=|\mathcal{F} \cap \{\emptyset, X\}|$.
So it is enough to consider the following cases.

\vspace{1mm}

\textit{Case I.} 
In this case we assume that 
$|\mathcal{F}| = |ess_{\mathcal{F}}(X)|+1$.
By equivalence of parts $\ref{pathdenseandleftinequisequ}$ and $\ref{GFistreealllabelsaredifferent}$ of Theorem \ref{characterizeextvc=1withequalities,right=impliesleft=},
$G_{\mathcal{F}}$ is a labelled directed tree with distinct labels.
Now by Lemma \ref{selfanddualpathdensefortreeandsemitree}(\ref{1-inctreesystemsdualpathdenseiffarepathorstarlike}) and the assumption that the dual of our system is well-graded, it follows that the system is either a full-chain system
or an upward-starlike system or a downward-starlike system.

\vspace{1mm}

\textit{Case II.} 
In this case we assume that $|\mathcal{F}| = |ess_{\mathcal{F}}(X)|+2$.
Now by using Theorem $\ref{structureofpathdensewithF=ess+2} (\ref{pathdensewithF=ess+2} \Rightarrow \ref{1-incgraphissemitree}$)
and Lemma $\ref{selfanddualpathdensefortreeandsemitree} (\ref{dualofsemitreeisnotpathdense})$, the dual system cannot be a well-graded family. So this case contradicts the assumptions and is impossible to happen.

\vspace{1mm}

\textit{Case III.} 
In this case we assume that $|\mathcal{F}| = |ess_{\mathcal{F}}(X)|+3$.
Hence, by Lemma \ref{dualselfextpathdenseimpliesalmostequality}(\ref{selfanddualPartialcubeinequ}), we have 
$\emptyset, X \in \mathcal{F}$. So every one element subset of $X$ is shattered by $\mathcal{F}$. Thus, $X=ess_{\mathcal{F}}(X)$. Therefore, $|\mathcal{F}|= |X|+3$ and since $|V(G_{\mathcal{F}})|=|\mathcal{F}|$, we have $|V(G_{\mathcal{F}})|=|X|+3$.
Since the system is a well-graded family, there is a one-way path $P$ in $G_{\mathcal{F}}$ from $v_{\emptyset}$ to $v_X$ with $|X|+1$ vertices. 
Hence, there are exactly two vertices of $G_{\mathcal{F}}$, say $v_B$ and $v_C$ for some $B,C \in \mathcal{F}$, which are outside of the path $P$.
In particular, $B$ and $C$ are different from $\emptyset$ and $X$.
Let $n:=|X|$. Without loss of generality, we may assume that $X=\{a_1,a_2,\ldots,a_n\}$ and vertices of the path $P$ consists of $v_{A_i}$'s (for $i=0,\dots,n$) where $A_0=\emptyset$ and for each $1 \leqslant i \leqslant n$, $A_i:=\{a_j:1 \leqslant j \leqslant i\}$ and $A_i$'s are some members of $\mathcal{F}$. 
So we have $\mathcal{F}=\{A_i:0 \leqslant i\leqslant n\}\cup \{B,C\}$.
Note that for each $1 \leqslant i \leqslant n$, $a_i$ is the label of the edge between $v_{A_{i-1}}$ and $v_{A_i}$ in $P$.
Without loss of generality, assume that $|B| \leqslant |C|$. 
Since the system 
is a well-graded family, there is a one-way path in $G_{\mathcal{F}}$ with $|B|+1$ vertices from $v_{\emptyset}$ to $v_B$. 
So there is a $E \in \mathcal{F}$ such that 
$E \subset B$ and $B$ has exactly one element more than $E$. Since $|E|<|B|\leqslant |C|$, $E$ is different from $B$, $C$ and $X$.
Thus, $E=A_{i_0}$ for some $0 \leqslant i_0 \leqslant n-1$. 
Let $1 \leqslant t_0 \leqslant n$ be such that $a_{t_0}$ is the single element of $B \setminus A_{i_0}$. 
So $v_{A_{i_0}}$ is adjacent to $v_B$ with an edge with label $a_{t_0}$.
Note that $t_0 > i_0$ since otherwise $a_{t_0} \in \{a_j:1 \leqslant j \leqslant i_0\}= A_{i_0}$.
Since $v_{A_{i_0+1}}$ is a vertex of the path $P$ and $v_B$ is not in $P$, $B$ is different from $A_{i_0+1}$. Hence, we have $t_0 \not = i_0+1$ since otherwise
$B=A_{i_0} \cup \{a_{t_0}\}=A_{i_0} \cup \{a_{i_0+1}\}=A_{i_0+1}$. So $t_0 \geqslant i_0+2$.

\vspace{0.75mm}

Now consider the dual system of $\mathcal{F}$.
As explained in Remark \ref{inpartcubedgeswithlabelaiscutsets}, for every $1 \leqslant r \leqslant n$, the set of edges of $G_{\mathcal{F}}$ with label $a_r$ is exactly the cut-set of the edges between $\mathcal{A}_{a_r}^c$ and $\mathcal{A}_{a_r}$ where $\mathcal{A}_{a_r}$ is viewed as a subsets of the vertices of $G_{\mathcal{F}}$ in the way explained in Remark \ref{identifyingY_Fwithsetofverticesof1-inc}. We mention the following facts and use them later.

\vspace{1mm}

(a) For any $1 \leqslant m,r \leqslant n$, vertex $v_{A_m}$ of path $P$ 
belongs to $\mathcal{A}_{a_r}$ if and only if $a_r \in A_m$ if and only if $r \leqslant m$.

\vspace{0.75mm}

(b) Since $a_{t_0} \in B$, we have $v_B \in \mathcal{A}_{a_{t_0}}$. 

\vspace{0.75mm}

(c) Since $B=A_{i_0} \cup \{a_{t_0}\}$, for any $r >i_0$ with $r \not =t_0$, we have $a_r \not \in B$ which follows that $v_B \not \in \mathcal{A}_{a_r}$.

\vspace{1mm}

Now consider $\mathcal{A}_{a_{t_0}}$ as a member of the dual system and analyse it.  
Since the dual system contains other elements except $\mathcal{A}_{a_{t_0}}$ and also is assumed to be a well-graded family, it should contain some $\mathcal{H}$ such that $|\mathcal{A}_{a_{t_0}} \triangle \mathcal{H}|=1$. 
So $\mathcal{H}=\mathcal{A}_{a_s}$ for some $s \not = t_0$.
We divide the situations for $s$ to the following three cases.

\vspace{1mm}

$\bullet$ If $s \leqslant i_0$, then 
by above statement (a), we have
$v_{A_{i_0}},v_{A_{i_0+1}} \in \mathcal{A}_{a_s}$
while since $t_0 \geqslant i_0+2$, 
$v_{A_{i_0}},v_{A_{i_0+1}} \not \in \mathcal{A}_{a_{t_0}}$.

\vspace{1mm}

$\bullet$ If $i_0 < s < t_0$, then by statements (b) and (c), we have $v_B \in \mathcal{A}_{a_{t_0}} \setminus \mathcal{A}_{a_s}$.
Also by statement (a) we have $v_{A_{t_0-1}} \in \mathcal{A}_{a_s} \setminus \mathcal{A}_{a_{t_0}}$.

\vspace{1mm}

$\bullet$ If $s > t_0$, then by statements (b) and (c), we have $v_B \in \mathcal{A}_{a_{t_0}} \setminus \mathcal{A}_{a_s}$. Also by statement (a) we have $v_{A_{t_0}} \in \mathcal{A}_{a_{t_0}} \setminus \mathcal{A}_{a_s}$.

\vspace{1mm}

In each of the above cases we have $|\mathcal{A}_{a_{t_0}} \triangle \mathcal{A}_{a_s}|>1$. 
So such $\mathcal{H}$ does not exist which is a contradiction. Therefore, case III is impossible to happen.

\vspace{1mm}

So far we have shown that only Case I is possible to happen and in that case we have $|\mathcal{F}|=|ess_{\mathcal{F}}(X)|+1$. 
It follows that every self-and-dual well-graded family is either a full-chain system or an upward-starlike or a downward-starlike system and moreover, it has additionality 1.
This completes the proof of part \ref{selfanddualpathdenseiffchainorstrongstarlike}.

\vspace{2mm}

\ref{selfedddualpathdenseiffchain})
If the system is a full-chain system, 
then it is easy to see that it is a self-and-ess-dual well-graded family. 
Now we prove the converse. Since the system is a self-and-ess-dual well-graded family, by Lemma \ref{dualselfextpathdenseimpliesalmostequality}(\ref{selfandessdualPartialcubeinequ}) we have 
$|\mathcal{F}|=|ess_{\mathcal{F}}(X)|+1$. 
So the additionality of the system is 1.
Therefore, since the system is a well-graded family, by 
$G_{\mathcal{F}}$ is a labelled directed tree with distinct labels.
Now by Lemma \ref{selfanddualpathdensefortreeandsemitree}(\ref{1-inctreesystemsessdualpathdenseiffarepath}) and the assumption that the ess-dual of the system is well-graded, it follows that the system is a full-chain system. 
\hfill $\square$
\end{proof}

\vspace{1mm}

The following statement characterizes the self-and-dual extremal (maximum) systems.

\begin{cor}\label{characterizationselfanddualextremalandmaximum}
\begin{enumerate}
\item{A set system is self-and-dual extremal if and only if it is a self-and-dual well-graded family if and only if it is of one of three types of systems mentioned in Theorem \ref{characterizationselfanddualpathdense}(\ref{selfanddualpathdenseiffchainorstrongstarlike}).}\label{characterizationselfanddualextremal}
\item{A set system $(X,\mathcal{F})$ is self-and-dual maximum if and only if either $|\mathcal{F}|=1$ or $\mathcal{F}=\{\emptyset,X\}$.
}\label{characterizationselfanddualmaximum}
\end{enumerate}
\end{cor}
\begin{proof}
\ref{characterizationselfanddualextremal}) By Remark \ref{Grecohamdisequalsgraphdis}, every self-and dual extremal system is a self-and dual well-graded family. Also by Theorem \ref{characterizationselfanddualpathdense}, every self-and dual well-graded system is either a full-chain system or an upward-starlike or a downward-starlike system. Moreover, by Remark \ref{factsaboutfullchainupwarddownwardstarlike}(\ref{everyupordownstarlikeisselfanddualextremal}), such systems are self-and-dual extremal.

\ref{characterizationselfanddualmaximum})
Let $(X,\mathcal{F})$ be a set system.
If $|\mathcal{F}|=1$, namely $\mathcal{F}=\{A\}$ for some $A \subseteq X$, then
the system is $0$-maximum and it is easy to see that its dual system is a full-chain system of length 1 (if $A=\emptyset$ or $A=X$) or 2 (if $\emptyset \subsetneq A \subsetneq X$)
on a one-element domain. So the dual system is a maximum system of VC-dimension 0 or 1.
If $\mathcal{F}=\{\emptyset,X\}$, then it is easy to see that the system is self-and dual maximum.
For the converse, we assume that $(X,\mathcal{F})$ is a self-and dual maximum system with $|\mathcal{F}|\geqslant 2$. It is enough to show that $\mathcal{F}=\{\emptyset,X\}$.
By remarks \ref{maximumsareextremal} and \ref{Grecohamdisequalsgraphdis}, every self-and dual maximum system is a self-and dual well-graded family which in turn, by Theorem \ref{characterizationselfanddualpathdense}, is either a full-chain system, an upward-starlike system or a downward-starlike system. 
But by Remark \ref{factsaboutfullchainupwarddownwardstarlike}(\ref{noupwarddownwardstarlikeismaximum}), the two latter cases cannot happen and so the system is a full-chain system. Also since $|\mathcal{F}|\geqslant 2$, the length of the chain is at least $2$.
Hence, by Remark \ref{factsaboutfullchainupwarddownwardstarlike}(\ref{fullchainlengthatleast2maximumiff}), $X=ess_{\mathcal{F}}(X)$ and the first member of the chain is $\emptyset$.
Now it is not hard to see that the dual system $(Y_{\mathcal{F}},\mathcal{F}^*)$ is a full-chain system with 
$|\mathcal{F}^*|=|\mathcal{F}|-1$ which starts with a nonempty set. So by assumption that the dual system $(Y_{\mathcal{F}},\mathcal{F}^*)$ is a maximum system and using Remark \ref{factsaboutfullchainupwarddownwardstarlike}(\ref{fullchainlengthatleast2maximumiff}) for  $(Y_{\mathcal{F}},\mathcal{F}^*)$, we have $|\mathcal{F}^*| \leqslant 2$.
Thus, $|\mathcal{F}^*| =1$ which follows that
$|\mathcal{F}|=2$. Therefore, since $X=ess_{\mathcal{F}}(X)$ and $\emptyset \in \mathcal{F}$, it is easy to see that $\mathcal{F}=\{\emptyset,X\}$.
\hfill $\square$
\end{proof}

\begin{rem}
The main results of the paper so far, namely theorems
\ref{characterizeextvc=1withequalities,right=impliesleft=}, 
\ref{structureofpathdensewithF=ess+2} and \ref{characterizationselfanddualpathdense}, hold for the set systems $(X,\mathcal{F})$ which $X$ is an infinite domain but $ess_{\mathcal{F}}(X)$ is finite. 
\end{rem}

\section{Well-gradedness of set systems associated to graphs}\label{partcubeofsetsystemsassociatedtographs}

Various kinds of set systems associated to graphs and certain combinatorial features of them such as their VC-dimension have been investigated in papers such as \cite{KozmaMoranShatteringGraphorientation}
and \cite{KranakisKrizancWoeginger}. 
In this final section, we investigate the properties of well-gradedness, being extremal and being maximum for some of the set systems associated to graphs.
More precisely, in subsections \ref{graphswithpartialcubeneighbourhoodsystems} and \ref{graphswithpartialcubeclosedneighbourhoodsystems}, we use the results of the previous sections in order to characterize the property of being well-graded, extremal or maximum for the set systems of neighbourhoods and closed-neighbourhoods of the graphs.
Moreover, in Subsection \ref{graphswithpartialcubeCliqueIndepsystems}, we do the similar things for the set systems of cliques and 
independent sets of graphs and also investigate the VC-dimension of such systems.

\subsection{Graphs with well-graded families of neighbourhoods}\label{graphswithpartialcubeneighbourhoodsystems}

Let $G$ be an undirected graph, possibly with loops (i.e. edges connecting a vertex to itself).
For any vertices $u$ and $v$ of $G$, we use the notation $u \sim v$ for $u$ and $v$ are adjacent by an edge. Note that if there is a loop on a vertex $v$, then $v$ is adjacent to itself.
We denote the set of vertices adjacent to $v$ by $N_G(v)$ and call it the \textit{neighbourhood} of $v$. Also we denote by $N_G[v]$ the \textit{closed-neighbourhood} of $v$ which is $N_G(v) \cup \{v\}$. 
Let $\mathcal{N}(G):=\{N_G(v):v \in V(G)\}$ and $\mathcal{N}_{cl}(G):=\{N_G[v]:v \in V(G)\}$. We consider the set systems
$(V(G),\mathcal{N}(G))$ and $(V(G),\mathcal{N}_{cl}(G))$ and call them the \textit{neighbourhood-set system} and \textit{closed-neighbourhood-set system} of the graph $G$.
When it is clear from the context, we usually write $N(v)$, $N[v]$, $\mathcal{N}$ and $\mathcal{N}_{cl}$ instead of $N_G(v)$, $N_G[v]$, $\mathcal{N}(G)$ and $\mathcal{N}_{cl}(G)$ respectively.
We call two vertices $v, w$ of $G$ \textit{twin} (\textit{closed-twin}) if $N(v)=N(w)$ ($N[v]=N[w]$). We call a graph \textit{twin-free} (\textit{closed-twin-free}) if it does not have twin (closed-twin) vertices.
We call a vertex of a graph \textit{fully-connected} if it is connected to every other vertex of the graph.
A graph is called \textit{semi-twin-free} if whenever 
two distinct vertices are twin, then both are isolated, where by an isolated vertex we mean a vertex of degree 0.
Similarly, we call a graph \textit{semi-closed-twin-free} if whenever 
two distinct vertices are closed-twin, then both are fully-connected.
Obviously every twin-free graph is semi-twin-free and every closed-twin-free graph is semi-closed-twin-free. We recall from Definition \ref{defdualsystem} that a self-dual system is a system which is isomorphic to its dual system.
Also as defined in Definition \ref{defalmostselfdual}, the purification of any almost self-dual system is isomorphic to its dual system.

\begin{rem}\label{neighbourhoodsystemofgraphsareselfdual}
If $G$ is a 
twin-free undirected graph, then $(V(G),\mathcal{N})$, the neighbourhood-set system of $G$, is self-dual. In general, without having the twin-free assumption, $(V(G),\mathcal{N})$ is almost self-dual (see Definition \ref{defalmostselfdual}).
The same statements hold if one replaces neighbourhood-set system to closed-neighbourhood-set system, $\mathcal{N}$ to $\mathcal{N}_{cl}$ and twin-free to closed-twin-free. 
\end{rem}
\begin{proof}
We define the following map between $(V(G),\mathcal{N})$ and its dual system $(Y_{\mathcal{N}},\mathcal{N}^*)$:
$$f: V(G) \rightarrow Y_{\mathcal{N}}$$
$$v \rightarrow y_{N(v)}.$$
So for every $v \in V(G)$ we have 
$f(N(v))=\{f(w):w \in N(v)\}=\{y_{N(w)}:w \in N(v)\}=\{y_{N(w)}:v \in N(w)\}=\mathcal{A}_{v}.$
It is clear that $f$ is a surjective map on $Y_{\mathcal{N}}$ and moreover, if $G$ is twin-free, then $f$ is injective.
So in this case one can see that $f$ is an isomorphism between two systems which follows that the neighbourhood-set system of $G$ is self-dual.
In general, without having the twin-free assumption on $G$, for any two vertices $v$ and $w$ we have $f(v)=f(w)$ if and only if $v$ and $w$ are twin vertices and one can see that it is equivalent to say that $v \sim_{\mathcal{N}} w$ (as defined in Definition \ref{eqrelhavingsameFtype}).
So $f$ induces a map from $\frac{V(G)}{\sim_{\mathcal{N}}}$ to $Y_{\mathcal{N}}$ and one can see that it gives rise to an isomorphism between $(\frac{V(G)}{\sim_{\mathcal{N}}},\tilde{\mathcal{N}})$ and $(Y_{\mathcal{N}},\mathcal{N}^*)$ where we recall that $(\frac{V(G)}{\sim_{\mathcal{N}}},\tilde{\mathcal{N}})$ is the purification of the set system $(V(G),\mathcal{N})$ as defined in the Definition \ref{systemquotientofERsameFtype}.
Note that a similar proof works if one replaces neighbourhood-set system to closed-neighbourhood-set system, $N(v)$'s to $N[v]$'s, $\mathcal{N}$ to $\mathcal{N}_{cl}$ and twin-free to closed-twin-free. 
\hfill $\square$
\end{proof}

\begin{dfn}\label{def(closed)neighbWGExtMax}
We call a 
graph $G$ \textit{neighbourhood-well-graded} if the neighbourhood-set system of $G$ is a well-graded family.
Similarly, we call $G$ \textit{neighbourhood-extremal} (\textit{neighbourhood-maximum}) 
if the neighbourhood set system of $G$ is an extremal (maximum) set system.
We can also define the notions closed-neighbourhood-well-graded, closed-neighbourhood-extremal and closed-neighbourhood-maximum in a similar way by replacing neighbourhood-set system in the definitions by closed-neighbourhood-set system.
\end{dfn}

\begin{rem}\label{pathdenseimpliessemireduced}
Assume that $G=(V,E)$ is a loopless neighbourhood-well-graded graph. Then $G$ is semi-twin-free. Similarly, if $G$ is a closed-neighbourhood-well-graded graph, then it is semi-closed-twin-free.
\end{rem}
\begin{proof}
We simplify the notations 
and let $(V,\mathcal{N})$ to be the neighbourhood-set system of $G$.
Let $u$ and $w$ be twin vertices. So they have the same $\mathcal{N}$-types (as defined in Definition \ref{eqrelhavingsameFtype}).
Since $(V,\mathcal{N})$ is a well-graded family, by Remark \ref{pathdensex,ysametypeoutsideess} every two vertices with the same $\mathcal{N}$-types are in $V \setminus ess_{\mathcal{N}}(V)$. 
Thus, $u,w \in V \setminus ess_{\mathcal{N}}(V)$. 
So $u$ (and also $w$) either belongs to all $N(v)$'s or to none of them.
But since $u$ does not belong to all $N(v)$'s (for example $u \not \in N(u)$ since $G$ is loopless) then it is isolated. The same holds for $w$ too.
So $G$ is semi-twin-free.

Similarly, assume that $(V,\mathcal{N}_{cl})$ is the closed-neighbourhood-set system of $G$ and $u$ and $w$ are closed-twin vertices. So they have the same $\mathcal{N}_{cl}$-types and by well-gradedness
and Remark \ref{pathdensex,ysametypeoutsideess}, we have 
$u,w \in V \setminus ess_{\mathcal{N}_{cl}}(V)$. Since $u$ 
belongs to at least one member of $\mathcal{N}_{cl}$ (for example $N[u]$), then it belongs to all $N[v]$'s which follows that it is fully-connected. The same holds for $w$ too.
So $G$ is semi-closed-twin-free. \hfill $\square$
\end{proof}

\begin{rem}\label{neighbourhoodpathdenseimpliesselfanddualpathdense}
Assume that $G$ is a loopless neighbourhood-well-graded graph. Then the neighbourhood-set system of $G$ is a self-and-dual well-graded family. Similarly, if $G$ is closed-neighbourhood-well-graded, then its closed-neighbourhood-set system is a self-and-dual well-graded family.
\end{rem}
\begin{proof}
We use the same notations as Remark \ref{neighbourhoodsystemofgraphsareselfdual}.
Assume that $G$ is loopless and neighbourhood-well-graded. We must show that neighbourhood-set system of $G$, namely $(V(G),\mathcal{N})$, is dual well-graded. By Remark \ref{neighbourhoodsystemofgraphsareselfdual}, the dual system of $(V(G),\mathcal{N})$ is isomorphic to $(\frac{V(G)}{\sim_{\mathcal{N}}},\tilde{\mathcal{N}})$.  
So it is enough to show that $(\frac{V(G)}{\sim_{\mathcal{N}}},\tilde{\mathcal{N}})$ is a well-graded family.
By Remark \ref{pathdenseimpliessemireduced}, $G$ is semi-twin-free. 
So if $u \sim_{\mathcal{N}} w$ (which is equivalent to say that $u$ and $w$ are twin), then $u$ and $w$ are isolated.
So $\frac{V(G)}{\sim_{\mathcal{N}}}$ can be obtained from $V(G)$ by identifying all  isolated vertices in $V(G)$ to a single isolated vertex. Also 
$\tilde{\mathcal{N}}$ and $\mathcal{N}$ can be identified and seen as the same. 
Now, having these descriptions of $\frac{V(G)}{\sim_{\mathcal{N}}}$ and $\tilde{\mathcal{N}}$ and by using the assumption that $(V(G),\mathcal{N})$ is a well-graded family, it is easy to see that $(\frac{V(G)}{\sim_{\mathcal{N}}},\tilde{\mathcal{N}})$ is also a well-graded family. A similar argument shows that if $G$ is closed-neighbourhood-well-graded, then its closed-neighbourhood-set system is a self-and-dual well-graded family. \hfill $\square$
\end{proof}

\begin{dfn}\label{defHalfgraphs}
By the half-graph of order $n \geqslant 1$ we mean a bipartite graph with parts consisting of vertices $\{a_1,\ldots,a_n\}$ and $\{b_1,\ldots,b_n\}$ such that one of the following happens:

$\bullet$ For every $1 \leqslant i,j \leqslant n$, $a_i \sim b_j$ if and only if $i \leqslant j$, 

$\bullet$ For every $1 \leqslant i,j \leqslant n$, $a_i \sim b_j$ if and only if $i \geqslant j$.
\end{dfn}

Half-graphs are important combinatorial objects which appear in graph theory and some other fields of mathematics such as arithmetic combinatorics and logic (in particular model theory). 
Note that, as roughly explained in the introduction of the paper, half-graphs have strong connections with Shelah's notion of "stable theories" in the "classification theory" (see \cite{bookshelahclassification}) from model theory.
As another example of applications of half-graphs, one can mention the celebrated Szemeredi regularity lemma which is an important result in combinatorics. In the context of this theorem, 
many interesting connections between the notion of "irregular-pairs" in the theorem and half-graphs have been discovered in combinatorics. The interested reader can refer to for example \cite{ConlonFoxBoundsforgraphregularityandremovallemmas} and \cite{MalliarisShelahReglemforstablegraphs} to see some precise connections between half-graphs and Szemeredi regularity lemma.
In \cite{MalliarisShelahReglemforstablegraphs}, in fact half-graphs and the model theoretic order property are involved 
for analysing the Szemeredi regularity lemma. Such analysis gives rise to some improvements of the theorem for the case of stable graphs in that paper.
Figure 2 illustrates an example of a half-graph.
\begin{center}
\[\begin{tikzpicture}
    \vertex[fill] (b1) at (0,1) [label=above:] {};
    \vertex[fill] (b2) at (1,1) [label=above:] {};
%        \vertex[fill] (b2) at (1,1) [label=above:$b_{2}$] {};
	\vertex[fill] (b3) at (2,1) [label=above:] {};
	\vertex[fill] (b4) at (3,1) [label=above:] {};
	\vertex[fill] (b5) at (4,1) [label=above:] {};
	\vertex[fill] (a1) at (0,0) [label=below:] {};
	\vertex[fill] (a2) at (1,0) [label=below:] {};
	\vertex[fill] (a3) at (2,0) [label=below:] {};
	\vertex[fill] (a4) at (3,0) [label=below:] {};
	\vertex[fill] (a5) at (4,0) [label=below:] {};
	\path
		(a1) edge (b1)
		(a1) edge (b2)
		(a1) edge (b3)				
		(a1) edge (b4)		
    	(a1) edge (b5)
    	
    	(a2) edge (b2)
		(a2) edge (b3)				
		(a2) edge (b4)		
    	(a2) edge (b5)
    	
    	(a3) edge (b3)				
		(a3) edge (b4)		
    	(a3) edge (b5)
    	
    	(a4) edge (b4)		
    	(a4) edge (b5)
    	
    	(a5) edge (b5)
				
			;
\end{tikzpicture}\]
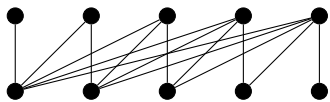
Figure 2: The half-graph of order 5.
\end{center}

We will need the following remark in the proof of the main theorem of this section.

\begin{rem}\label{inchainofneighbsnovibelongtoNvj}
Assume that $N(v_1)\subseteq \ldots \subseteq N(v_n)$ is an increasing chain of neighbourhoods of some vertices in a loopless graph. Then for any $1 \leqslant i,j \leqslant n$ we have $v_i \not \in N(v_j)$.
\end{rem}
\begin{proof} Assume for contradiction that $v_i \in N(v_j)$ for some $i$ and $j$. We consider two cases.
If $j \leqslant i$ then since $N(v_j) \subseteq N(v_i)$, we have $v_i \in N(v_i)$. It follows that $G$ has a loop which is a contradiction. 
Now assume that $i < j$. Since $v_i \in N(v_j)$ we have $v_j \in N(v_i)$ and since 
$N(v_i) \subseteq N(v_j)$ we have $v_j \in N(v_j)$ which again follows that $G$ has a loop which is a contradiction. 
\hfill $\square$
\end{proof}

\vspace{1mm}

The following theorem is the main result of this section and characterizes the neighbourhood-well-graded loopless graphs and relates them to the half-graphs.

\begin{thm}\label{Thmcharacterizeneighbourhoodpathdensegraphs}
Let $G$ be a 
loopless graph. Then $G$ is a neighbourhood-well-graded graph if and only if it is either a nonempty set of isolated vertices or a union of some disjoint half-graphs and a nonempty set of isolated vertices. 
\end{thm}
\begin{proof} 
$\Leftarrow$) Clearly, the neighbourhood-set system of every graph of the form of a nonempty set of isolated vertices consists of only one member $\emptyset$ and so is a well-graded family. 
Also It is not hard to see that for every graph of the form of a union of some disjoint half-graphs and a nonempty set of isolated vertices, the neighbourhood-set system is an upward-starlike system. Therefore, by Remark \ref{factsaboutfullchainupwarddownwardstarlike}(\ref{everyupordownstarlikesatisfiesTheorem}), the neighbourhood-set system of the graph is again a well-graded family.

\vspace{1mm}

$\Rightarrow$) Let $(V,\mathcal{N})$ be the neighbourhood-set system of $G$, 
where $V=V(G)$ and $\mathcal{N}$ is the family of neighbourhoods. Also assume that 
$(V,\mathcal{N})$ is a well-graded family.
So by Remark \ref{neighbourhoodpathdenseimpliesselfanddualpathdense}, $(V,\mathcal{N})$ is a self-and-dual well-graded family. 
Thus, by Theorem \ref{characterizationselfanddualpathdense}(\ref{selfanddualpathdenseiffchainorstrongstarlike}), $(V,\mathcal{N})$ is either a full-chain or an upward-starlike or a downward-starlike system and also its additionality is 1 which means that
$|\mathcal{N}|=|ess_{\mathcal{N}}(V)|+1$. 
Moreover, we have $|V| \geqslant |\mathcal{N}|$ since every member of $\mathcal{N}$ corresponds to (neighbourhood of) some vertex in $V$. 
Hence, $|V| \geqslant |ess_{\mathcal{N}}(V)|+1$. 
So $V \setminus ess_{\mathcal{N}}(V) \not = \emptyset$. 
Fix an arbitrary $v_0 \in V \setminus ess_{\mathcal{N}}(V)$. Thus, either $v_0$ belongs to every member of $\mathcal{N}$ or to no one.
Since $G$ is loopless, for every $v \in V$, $v \not \in N(v)$. So there is no $v \in V$ which belongs to every member of $\mathcal{N}$. Thus, $v_0$ does not belong to any member of $\mathcal{N}$. Hence, $v_0$ (and every other vertex in $V \setminus ess_{\mathcal{N}}(V)$) is an isolated vertex of $G$.
So $G$ contains a nonempty set of isolated vertices. Also $V \not \in \mathcal{N}$ which follows that $(V,\mathcal{N})$ is not a downward-starlike system.
Furthermore, we have $N(v_0) \in \mathcal{N}$ and $N(v_0)=\emptyset$ which implies that $\mathcal{N}$ contains $\emptyset$. 
Now we show that $(V,\mathcal{N})$ cannot be a full-chain system of length greater than $1$.
Assume for contradiction that $N(v_1) \subseteq \ldots \subseteq N(v_n)$ is a full-chain of elements of $\mathcal{N}$ and contains all members of $\mathcal{N}$ where $n \geqslant 2$. 
By Remark \ref{inchainofneighbsnovibelongtoNvj}, one concludes that 
no member of $\mathcal{N}$ contains $v_2$. It follows that $v_2$ is an isolated vertex. But this is a contradiction since by definition of a full-chain, $N(v_2) \not = \emptyset$.
Therefore, one sees that the only possibility for $(V,\mathcal{N})$ to be a full-chain system is that $\mathcal{N}=\{\emptyset\}$ and $V$ consists of a nonempty set of isolated vertices.

\vspace{1mm}

So far we have shown that $(V,\mathcal{N})$ is not a downward-starlike system and also is not a full-chain system unless $G$ is a set of isolated vertices.
So now we only need to consider the case that $(V,\mathcal{N})$ is an upward-starlike system and see what $G$ looks like in this case.
We recall from Definition \ref{defupwarddownwardstarlike} that each of the maximal 
full-chains starting from element $\emptyset$ in an upward-starlike system 
is called a wing of that system. 
For every wing $W$, by the length of $W$ 
we mean the number of members of $W$ except $\emptyset$. Also we call the set of elements of the domain of the set system which appear in the members of $W$ by the domain of $W$ and denote it by $D(W)$. 
Note that by definition of upward-starlike systems, labels of the edges of the one-inclusion graph of $(V,\mathcal{N})$ are distinct.
It follows that for every two different wings $W_1$ and $W_2$ of $(V,\mathcal{N})$, we have $D(W_1) \cap D(W_2)=\emptyset$.
For two different wings $W_1$ and $W_2$ of $(V,\mathcal{N})$, we say that $W_2$ accompanies $W_1$ if there are 
$\{x_1,\ldots,x_n\}$ and $\{y_1,\ldots,y_n\}$, vertices in $V$ and all distinct, such that members of $W_1$ consist of $\emptyset$ and $N(x_i)$'s, members of $W_2$ consist of $\emptyset$ and $N(y_i)$'s and moreover, for each $1 \leqslant i \leqslant n$, $N(x_i)=\{y_1,\ldots,y_i\}$ and $N(y_i)=\{x_i,\ldots,x_n\}$.
It is not hard to see that if a wing $W_2$ accompanies a wing $W_1$, then $W_1$ accompanies $W_2$ too.
So we say that $W_1$ and $W_2$ accompany each other if one accompanies the other one or equivalently each one accompanies the other one. We will see below that in fact accompanying pairs correspond to half-graph subgraphs of $G$.

\vspace{1mm}

Fix a wing $W$ of $(V,\mathcal{N})$ and assume that it consists of the increasing sets $\emptyset$, $N(v_1), \ldots, N(v_n)$ for some distinct $v_i$'s in $V$ and some $n$. Since $\emptyset \subseteq N(v_1) \subseteq \ldots \subseteq N(v_n)$ is a full-chain,
there are distinct $u_i$'s in $V$ such that $N(v_i)=\{u_1,\ldots,u_i\}$ for every $1 \leqslant i \leqslant n$. 
By using Remark \ref{inchainofneighbsnovibelongtoNvj}, we have $\{u_1,\ldots,u_n\} \cap \{v_1,\ldots,v_n\}=\emptyset$.
Also one can see that for every $1 \leqslant i \leqslant n$, we have $\{v_i,\ldots,v_n\} \subseteq N(u_i)$. 
We claim that $N(u_i)=\{v_i,\ldots,v_n\}$ for every $1 \leqslant i \leqslant n$. 
Assume for contradiction that there is some $h \in N(u_{i_0}) \setminus \{v_{i_0},\ldots,v_n\}$ for some $i_0$.
Note that $h$ belongs to some member(s) of $\mathcal{N}$ for example $N(u_{i_0})$, and does not belong to some other one(s) such as $N(h)$.
Thus, $h \in ess_{\mathcal{N}}(V)$. So by Remark \ref{pathdensex,ysametypeoutsideess},
no other member of $V$ has the same $\mathcal{N}$-type (see Definition \ref{eqrelhavingsameFtype}) as $h$.
Consequently, for every $v \in V$, we have $N(h) \not = N(v)$. 
In particular, $N(h) \not = N(v_i)$ for every $1 \leqslant i \leqslant n$.
Hence, $N(h)$ does not belong to the wing $W$.
It can be seen that since $(V,\mathcal{N})$ is an upward-starlike system, 
every member of $\mathcal{N}$ except $\emptyset$ appears in exactly one wing. 
So $N(h)$ is a member of a wing different from $W$.
Hence, since $u_{i_0}$ belongs to the domain of $W$ and the domains of every two different wings are disjoint (as mentioned earlier), $u_{i_0} \not \in N(h)$.
Thus, $h \not \in N(u_{i_0})$ which is a contradiction.
Therefore, $N(u_i)=\{v_i,\ldots,v_n\}$ for every $1 \leqslant i \leqslant n$.

\vspace{1mm}

There is a unique wing, say $W'$, containing $N(u_1)$ as a member.
Thus, $\{v_1,\ldots,v_n\} = N(u_1) \subseteq \ D(W')$.
So for every $2 \leqslant j \leqslant n$, we have $N(u_j) \subseteq D(W')$ since $N(u_j) \subseteq N(u_1)$. Hence, again since domains of different wings are disjoint and every member of $\mathcal{N}$ except $\emptyset$ appears in exactly one wing, for every $2 \leqslant j \leqslant n$, $N(u_j)$ must belong to the wing $W'$. 
Thus, $\emptyset \subseteq N(u_n) \subseteq N(u_{n-1}) \subseteq \ldots \subseteq N(u_1)$ are some starting members (in the sense of order of inclusion) 
of the wing $W'$. We claim that $W'$ does not have any other member. Assume for contradiction that there exists some $z$ different from $u_i$'s such that $N(z) \in W'$ and $N(z)$ is different from $N(u_i)$'s. 
It is easily seen that $N(z)$ appears after $N(u_1)$ in the order of   inclusions in the wing $W'$ and $N(u_1) \subset N(z)$.
Thus, $v_1 \in N(z)$ (and also any other $v_i$ for $i \leqslant n$).
So $z \in N(v_1)=\{u_1\}$. Hence, $z=u_1$ which is a contradiction with our assumption that $z$ is different from $u_i$'s. Consequently, $W'$ has no more members than what mentioned above and the lengths of $W$ and $W'$ are the same.
Now one sees that $W$ and $W'$ accompany each other.
Also it is easy to see that any wing different from $W$ and $W'$ does not accompany $W$ or $W'$.

\vspace{1mm}

So far, it is shown that starting with any wing, one finds a unique different wing in such a way that they two accompany each other and moreover, 
any other wing accompanies none of them.
Using this, it is not hard to observe that the family of wings of $\mathcal{N}$ can be partitioned into some pairs of wings where in each pair, two wings accompany each other.
Furthermore, it is easy to verify that if one ignores the element $\emptyset$ in the wings,
then these pairs of wings form a partitioning of $\mathcal{N}\setminus \{\emptyset\}$.
It is worth mentioning that although the length of both wings in each pair of accompanying wings are the same, wings of different pairs can have different lengths.
It is easy to see that each pair $(W,W')$ of accompanying wings in the mentioned partitioning, corresponds to a half-graph in $G$ with the order equal to the length of the wings in that pair.
More precisely, ignoring the element $\emptyset$ of $W$ and $W'$, the other members of $W$ and $W'$ in fact form the family of neighbourhoods of a subset of vertices of $G$ such that the induced subgraph on that subset is a half-graph and also every vertex of that half-graph is not adjacent to any vertex outside of that half-graph.
Now it is not difficult to use the above mentioned partitioning of the family of wings to pairs of accompanying wings and conclude that $G$ is the union of some disjoint half-graphs and a nonempty set of isolated vertices.
\hfill $\square$
\end{proof}

\begin{cor}\label{characterizationneighbextremalgraphs}
A loopless graph is neighbourhood-extremal if and only if it is neighbourhood-well-graded if and only if it is of the form described in Theorem \ref{Thmcharacterizeneighbourhoodpathdensegraphs}.
\end{cor}
\begin{proof}
By using Theorem \ref{Thmcharacterizeneighbourhoodpathdensegraphs} and the fact that every extremal system 
is a well-graded family (see Remark \ref{Grecohamdisequalsgraphdis}), it is enough to show that every graph of the form of a nonempty set of isolated vertices or a union of some disjoint half-graphs and a nonempty set of isolated vertices is neighbourhood-extremal. 
But this is also clear since the neighbourhood-system of each of such graphs is either $\{\emptyset\}$ (in the case that $G$ is a nonempty set of isolated vertices) or an upward-starlike system which by Remark \ref{factsaboutfullchainupwarddownwardstarlike}(\ref{everyupordownstarlikesatisfiesTheorem}) is seen that is extremal. \hfill $\square$
\end{proof}

\begin{cor}\label{characterizationneighbmaximumgraphs}
A loopless graph is neighbourhood-maximum if and only if it is a nonempty set of isolated vertices.
\end{cor}
\begin{proof}
Clearly a nonempty set of isolated vertices is neighbourhood-maximum. For the converse, assume that $G=(V,E)$ is a neighbourhood-maximum graph. Then by using remarks \ref{maximumsareextremal} and \ref{Grecohamdisequalsgraphdis}, it is neighbourhood-well-graded and by Theorem \ref{Thmcharacterizeneighbourhoodpathdensegraphs}, it is of the form of a nonempty set of isolated vertices or a union of some disjoint half-graphs and a nonempty set of isolated vertices.
Note that each isolated vertex of $G$ cannot be shattered (as a single element subset of $V$) by the neighbourhood-set system. Therefore, since the neighbourhood-set systems is maximum, its VC-dimension should be zero. It follows that there is no half-graph in $G$ and $G$ is of the form of just a nonempty set of isolated vertices since otherwise, one can see that the VC-dimension of the neighbourhood-set system would be larger than zero.
\hfill $\square$
\end{proof}

\subsection{Graphs with well-graded families of closed-neighbourhoods}\label{graphswithpartialcubeclosedneighbourhoodsystems}

\begin{dfn}
A co-half-graph of order $n \geqslant 1$ 
is a graph with vertices $a_1,\ldots,a_n$ and $b_1,\ldots,b_n$ such that $a_i \sim b_j$ for every $i > j$,
$a_i \sim a_j$ and $b_i \sim b_j$ for every $i,j$ with $i \not = j$ and these are the only edges. 
\end{dfn}

For a loopless graph $G$, by $G^c$ we mean a loopless graph with the same vertex set such that two distinct vertices are connected if and only if they are not connected in $G$.
We call $G^c$ the \textit{complement} of $G$.
One can see that for any $n \geqslant 1$, the co-half-graph of order $n$ is exactly the complements of the half-graph of order $n$.

\begin{dfn}
By the \textit{between-full-union} of to graphs $G_1$ and $G_2$ we mean a graph $G_3$ with vertex set consisting of disjoint union of $V(G_1)$ and $V(G_2)$ and $E(G_3)=E(G_1) \cup E(G_2) \cup \{(v,w):v \in V(G_1), w \in V(G_2)\}$.
\end{dfn}

The following theorem characterizes closed-neighbourhood-well-graded, closed-neighbourhood-extremal and closed-neighbourhood-maximum graphs as defined in Definition \ref{def(closed)neighbWGExtMax}.

\begin{thm}\label{Thmcharacterizeclosedneighbourhoodpathdensegraphs}
Assume that $G$ is a 
loopless graph. 
Then $G$ is a closed-neighbourhood-well-graded graph if and only if it is a closed-neighbourhood-extremal graph if and only if it is either a complete graph
or a between-full-union of some co-half-graphs and a complete graph. 
Also all of these are equivalent to say that $G^c$ is a neighbourhood-well-graded graph.
Moreover, $G$ is closed-neighbourhood-maximum if and only if 
$G^c$ is neighbourhood-maximum if and only if $G$ is a complete graph.
\end{thm}
\begin{proof}
First we note that for every vertex $v$ of $G$ we have $N_G[v]=(N_{G^c}(v))^c$. So the set system $(V(G),\mathcal{N}_{cl}(G))$ 
is the same as the set system 
$(V(G),\overline{\mathcal{F}})$ (see Definition \ref{defFbarreverse}) where $\mathcal{F}=\mathcal{N}(G^c)$. 
Note that by Remark 
\ref{complementofpartcubeispartcube}, the system 
$(V(G),\overline{\mathcal{F}})$ is a well-graded family, extremal system or maximum system if and only if the system $(V(G),\mathcal{F})$ is a well-graded family, extremal system or maximum system respectively.
So $G$ is closed-neighbourhood-well-graded (closed-neighbourhood-extremal) if and only if $G^c$ is neighbourhood-well-graded (neighbourhood-extremal). 
But by Corollary \ref{characterizationneighbextremalgraphs},
$G^c$ is neighbourhood-well-graded if and only if $G^c$ is neighbourhood-extremal if and only if $G^c$ is either a nonempty set of isolated vertices or a union of some disjoint half-graphs and a nonempty set of isolated vertices.
So $G$ is closed-neighbourhood-well-graded if and only if it is closed-neighbourhood-extremal if and only if it is the complement of 
such mentioned graphs
or equivalently is either a complete graph or a between-full-union of some co-half-graphs and a complete graph.

Similarly, $G$ is closed-neighbourhood-maximum if and only if $G^c$ is neighbourhood-maximum, which is, by Corollary \ref{characterizationneighbmaximumgraphs}, equivalent to say that $G^c$ is a nonempty set of isolated vertices. It follows that $G$ is closed-neighbourhood-maximum if and only if it is a complete graph.
\hfill $\square$
\end{proof}

\subsection{Graphs with well-graded families of cliques and independent sets}\label{graphswithpartialcubeCliqueIndepsystems}

In this subsection, we consider the set systems associated to graphs obtained by cliques and independent sets of the graphs. 
By a \textit{clique} (\textit{independent set}) in a graph $G$, we mean a subset of vertices of $G$ such that every two distinct vertices of them are adjacent (non-adjacent) in $G$.
We consider $\emptyset$ and every single element subsets of the vertices of $G$ both as a clique and an independent set.
By the \textit{clique-number} (\textit{independence-number}) of $G$, we mean the size of the largest clique (independent set) of $G$.

\begin{dfn}
Let $G=(V,E)$ be a graph. We call a set system $(V,\mathcal{F})$, the set system of cliques (independent sets) of $G$ if elements of $\mathcal{F}$ are the cliques (independent sets) of $G$ where we remind that we consider $\emptyset$ and every single element subsets of $V$ both as a clique and an independent set.
\end{dfn}

We call a set system $(X,\mathcal{F})$ \textit{down-closed} if for every $A \in \mathcal{F}$ and $B \subseteq A$, we have $B \in \mathcal{F}$. One can see that every down-closed set system is an extremal system (and so a well-graded family).
Also it is easy to verify that the VC-dimension of every down-closed set system is the same as the size of the largest member of the system.

\begin{rem}
Let $G=(V,E)$ be a graph and $(V,\mathcal{F})$ be the set system of cliques (independent sets) of $G$. Then the following hold.
\begin{enumerate}
\item{The set system $(V,\mathcal{F})$ is down-closed and therefore an extremal system.
}\label{cliquesystemispartialcube}
\item{The VC-dimension of $(V,\mathcal{F})$ is equal to the clique-number (independence-number) of the graph $G$.}\label{VCofcliquesystemiscliquenumber} 
\end{enumerate} 
\end{rem}
\begin{proof}
\ref{cliquesystemispartialcube}) Assume that $A \in \mathcal{F}$. So $A$ is a clique (independent set) of $G$. Thus, every subset $B \subseteq A$ is also a clique (independent set) of $G$. Hence, $B \in \mathcal{F}$. Therefore, 
$(V,\mathcal{F})$ is down-closed and so an extremal system. 

\vspace{1mm}

\ref{VCofcliquesystemiscliquenumber})
Since $(V,\mathcal{F})$ is down-closed, its VC-dimension is equal to the size of the largest member of $\mathcal{F}$ which is the same as the clique-number (independence-number) of $G$. 
\hfill $\square$
\end{proof}

\vspace{1mm}

It is not very hard to see that for a graph $G$ with clique-number (independence number) at least 2, the set system of cliques (independent sets) is maximum if and only if $G$ is a complete graph (a set of isolated vertices).

\vspace{2.5mm}

\noindent {\bf Acknowledgement}
\vspace{2.5mm}

The author is indebted to Institute for Research in Fundamental Sciences, IPM, for support.

\end{document}